\documentclass[12pt,a4paper]{amsart}

\usepackage{mathtools}
\usepackage{xcolor}
\usepackage{amsfonts}
\usepackage{amssymb}
\usepackage{amsthm}
\usepackage{tikz}
\usetikzlibrary{arrows,positioning,automata,shadows,fit,shapes,shapes.multipart,
decorations.pathreplacing}
\usepackage{ytableau}
\usepackage[makeroom]{cancel}

\usepackage{hyperref}

\tikzset{
	state/.style={
		rectangle split,
		rectangle split parts=2,
		rectangle split part fill={red!30,blue!20},
		rounded corners,
		draw=black, very thick,
		minimum height=2em,
		text width=3cm,
		inner sep=2pt,
		text centered,
	}
}

\newcommand\abs[1]{\left|#1\right|}
\newcommand*{\aA}{\mathcal{A}}
\newcommand*{\bB}{\mathfrak{R}}
\newcommand*{\cC}{\mathfrak{C}}
\newcommand*{\dD}{\mathfrak{D}}
\newcommand*{\eE}{\mathcal{E}}

\newcommand*{\ldD}{\mathfrak{D}^\ell}
\newcommand*{\rdD}{\mathfrak{D}^r}
\newcommand*{\dd}{\mathfrak{d}}
\newcommand*{\ldd}{\mathfrak{d}^{\ell}}
\newcommand*{\rdd}{\mathfrak{d}^{r}}

\newcommand*{\bell}{{\smash{\mathsf{lps}}}}
\newcommand*{\belr}{{\smash{\mathsf{rps}}}}
\newcommand*{\rR}{\mathcal{R}}

\newcommand*{\bellcong}{\equiv_\bell}
\newcommand*{\belrcong}{\equiv_\belr}
\newcommand*{\belxcong}{\equiv_{{\smash{\mathsf{xps}}}}}
\newcommand*{\bellcongn}{\equiv_{\bell_n}}
\newcommand*{\belrcongn}{\equiv_{\belr_n}}
\newcommand*{\dstd}{\mathrm{std}^{-1}}
\newcommand*{\stdl}{\mathrm{std}_\ell}
\newcommand*{\stdr}{\mathrm{std}_r}
\newcommand*{\std}{\mathrm{std}}
\newcommand*{\Dstd}{\mathrm{Dstd}}
\newcommand*{\Stdl}{\mathrm{Std}_\ell}
\newcommand*{\Stdr}{\mathrm{Std}_r}
\newcommand*{\Std}{\mathrm{Std}}

\def \lps {lPS}
\def \rps {rPS}
\def \ps {PS}

\newcommand*{\Sh}{\mathrm{Sh}}

\DeclarePairedDelimiter{\parens}{\lparen}{\rparen}
\newcommand*{\evlit}{{\mathrm{ev}}}
\newcommand*{\ev}[2][]{\evlit\parens[#1]{#2}}

\theoremstyle{plain}
\newtheorem{thm}{Theorem}[section]
\newtheorem{lem}[thm]{Lemma}
\newtheorem{prop}[thm]{Proposition}
\newtheorem{cor}[thm]{Corollary}

\theoremstyle{definition}

\newtheorem{algorithm}[thm]{Algorithm}

\newtheorem{qst}[thm]{Question}

\title{The monoids of the Patience sorting algorithm}

\author{Alan J. Cain}
\address{Centro de Matem\'atica e Aplica\c{c}\~oes, Faculdade de Ci\^encias e
Tecnologia, Universidade Nova de Lisboa,
	2829-516 Caparica, Portugal}
\email{a.cain@fct.unl.pt}
\thanks{The first author was supported by
an Investigador FCT fellowship (IF/01622/2013/CP1161/CT0001).}

\author{Ant\'onio
Malheiro}
\address{Centro de Matem\'atica e Aplica\c{c}\~oes and Departamento de
	Matem\'atica, Faculdade de Ci\^encias e Tecnologia, Universidade Nova de Lisboa,
	2829-516 Caparica, Portugal}
\email{ajm@fct.unl.pt}
\thanks{For the first two authors,
this work was partially supported by the Funda\c{c}\~ao para a
Ci\^encia e a Tecnologia (Portuguese Foundation for Science and Technology)
through the project UID/MAT/00297/2013 (Centro de Matem\'atica e
Aplica\c{c}\~oes), and the project PTDC/MHC-FIL/2583/2014}

\author{F\'abio M. Silva}
\address{Departamento de Matem\'atica
	and CEMAT-CI\^ENCIAS, Faculdade de Ci\^encias, Universidade de Lisboa, Lisboa
	1749-016, Portugal.}
\email{femsilva@fc.ul.pt}
\thanks{The third author was supported by an FCT Lismath fellowship
(PD/BD/52644/2014) and partially supported by the FCT project CEMAT-Ci\^encias
UID/Multi/04621/2013}

\thanks{The authors would like to thank Prof. Gracinda Gomes for her kind advice and helpful observations during the
  writing of this paper.}

\begin{document}

\begin{abstract}
  The left patience sorting (\lps) monoid, also known in the literature as the Bell monoid, and the right patient
  sorting (\rps) monoid are introduced by defining certain congruences on words. Such congruences are constructed using
  insertion algorithms based on the concept of decreasing subsequences. Presentations for these monoids are given.

  Each finite-rank \rps\ monoid is shown to have polynomial growth and to satisfy a non-trivial identity (dependent on
  its rank), while the infinite rank \rps\ monoid does not satisfy a non-trivial identity. The \lps\ monoids of finite
  rank have exponential growth and thus do not satisfy non-trivial identities. The complexity of the insertion
  algorithms is discussed.

  \rps\ monoids of finite rank are shown to be automatic and to have recursive complete presentations. When the rank is
  $1$ or $2$, they are also biautomatic. \lps\ monoids of finite rank are shown to have finite complete presentations
  and to be biautomatic.
\end{abstract}


\maketitle
\tableofcontents

\section{Introduction}

Monoids arising from combinatorial objects have been intensively studied in recent years. Important examples include the
plactic (see \cite{MR646486, MR1905123, MR1464693}), the sylvester \cite{MR2081336, MR2142078}, the Chinese
\cite{MR1351284,MR1847182}, the hypoplactic \cite{Krob1997}, the Baxter \cite{MR2820726, MR2914637}, the stalactic
\cite{Hivert20081682}, and the Bell monoids \cite{Maxime07}. These monoids are associated to several important
combinatorial objects such as Young tableaux (plactic monoid), binary trees (sylvester monoid), Chinese staircases
(Chinese monoid) quasi-ribbon tableaux (hypoplactic monoid) or set partitions (Bell monoid), in the sense that their
elements can be identified with these objects. Their construction, in general, relies on two distinct approaches.

The first approach uses an insertion algorithm which computes a specific combinatorial object from a over a totally
ordered alphabet $\aA=\{1<2<\ldots\}$. Then the monoid is constructed by taking the quotient of the free monoid $\aA^*$
by the congruence that relates words that yield the same combinatorial object. For example, the plactic monoid can be
constructed as the quotient of the free monoid $\aA^*$ by the congruence that relates words that give rise to the same
semi-standard Young tableau under Schensted's insertion algorithm (see \cite[Chapter~5]{MR1905123}). The second method
consists of presenting a set of defining relations over $\aA^*$, which allows the construction of a congruence (which is
proven to coincide with the previous one), and again taking the quotient of $\aA^*$ over this congruence.

The richness of these monoids comes from the various perspectives from which we
can look at them, giving rise to many interesting questions. For example, in
general, to each such monoid we can associate a combinatorial Hopf algebra. Such
algebras include for instance the Poirier--Reutanauer algebra of tableaux
\textbf{FSym} arising from the plactic monoid \cite{poirier1995algebres,
duchamp2002noncommutative}, the planar binary tree algebra \textbf{PBT} of
Loday--Ronco obtained from the sylvester monoid \cite{MR2142078, Loday1998293},
the Hopf algebra of pairs of twin binary trees \textbf{Baxter} coming from the
Baxter  monoid \cite{MR2820726,MR2914637}, and the Hopf algebras \textbf{Sym}
and  \textbf{Bell} arising from the hypoplactic monoid \cite{Krob1997,
Novelli2000315} and the Bell monoid \cite{Maxime07}, respectively. However,
the importance of these monoids lies not just in the associated
combinatorial  algorithms or in the associated Hopf algebras, but also in the
fact that they are related to partial orders like the Tamari order and to Robinson--Schensted-like
correspondences \cite{MR2914637}.

We note that the plactic monoid appears in the literature defined in two different ways. When generalizing Schensted's
insertion algorithm from permutations to words, one can use either row-insertion or column-insertion. Column-insertion
is equivalent to row-insertion into a modified `Young tableau' where the entries in each row are required to be strictly
increasing and the entries in each column are required to be weakly decreasing. This gives rise to two different
anti-isomorphic `plactic monoids', one of which is more commonly used in combinatorics \cite[Chapter~5]{MR1905123} and
the other in representation theory and the theory of crystals
\cite{MR1358301,date1990representations,kashiwara1995crystal}..

The same issue appears when trying to generalize to words the Patience Sorting (\ps) algorithm described for
permutations in \cite{BL2005}.  The \ps\ monoids are the central subject of this paper and arise from the following two
possible generalizations: when considering the insertion of a symbol into a \ps\ tableaux we may allow entries to
columns to be weakly decreasing or strictly decreasing \cite{MR1694204}. This choice leads to the \rps\ monoid in the
first case and to the \lps\ monoid in the second. The \lps\ monoid is the `Bell monoid' considered in
\cite{Maxime07}. Contrary to the plactic monoid case, these generalizations give rise to two monoids which, as we will
see, are not anti-isomorphic, and, in fact, have very different properties. Section~\ref{sec:pstableaux} studies the
insertion algorithms corresponding to these two types of tableaux and shows how they give rise to congruences.


The plactic monoid of rank $n$ has polynomial growth of degree $(n^2+n)/2$ \cite{MR1351284}. It is also known that
plactic monoids of rank at most $3$ satisfy non-trivial identities but that the infinite-rank plactic monoid does not
satisfy a non-trivial identity \cite{kubat2015identities,ckkmo_placticidentity}. It is an open question whether plactic
monoids of finite rank greater than $3$ satisfy non-trivial identities. Section~\ref{sec:growth} shows that the \lps\
monoids of rank greater than $1$ have exponential growth, by proving that free monoids (of rank greater than $1$) embed
into these monoids, whereas the \rps\ monoids of finite rank have polynomial growth.  As a consequence, we deduce that
unlike other monoids associated to combinatorial objects such as the hypoplactic, the sylvester, the Baxter, the
stalactic and others \cite{1611.04151}, the \lps\ monoids of rank greater than $1$ do not satisfy non-trivial
identities. We also prove that finite-rank \rps\ monoids satisfy identities but that the infinite-rank \rps\ monoid does
not.

Finite-rank plactic, Chinese, and hypoplactic monoids are all presented by finite complete rewriting systems and the
finite-rank sylvester monoids are presented by regular complete rewriting systems \cite{CGMrewriting,chen2008grobner,
  karpuz2010complete,MR3283708}. Section~\ref{sec:complete} shows that finite-rank \lps\ monoids are also presented by
finite complete rewriting systems. As a corollary, we deduce that these monoids satisfy the homological finiteness
condition $FP_\infty$. We also show that \rps\ monoids are presented by infinite complete rewriting systems.

Many of the monoids mentioned above, such as finite-rank plactic, Chinese, hypoplactic, and sylvester monoids, are
biautomatic. Automaticity, introduced by Epstein et al. \cite{MR1161694} and later generalized to semigroups and monoids
\cite{MR1795250}, is a way of describing and computing with semigroups using finite-state automata. Section~\ref{sec4}
proves that finite-rank \rps\ monoids are automatic and that finite-rank \lps\ monoids are biautomatic. We deduce as a
corollary that the \rps\ and \lps\ monoids have word problem solvable in quadratic time.

\section{Preliminaries and notation}

In this section we introduce the notions that we shall use along the paper.
For more details concerning these constructions see for instance
\cite{MR1905123} and \cite{MR2142078} for words and standardization, and
\cite[Subsection~1.6]{howie1995fundamentals} and \cite{MR1215932} for presentations and
rewriting systems.

\subsection{Words and standardization}
\label{alphabetswords}

In what follows, let  $\aA=\{1 < 2 < 3<\ldots \}$ be the set of natural numbers
viewed as an infinite well-ordered  alphabet and let $\aA^*$ be the corresponding free monoid. Furthermore, for any $n\in \mathbb{N}$, denote by  $\aA_n$ the totally ordered subset of $\aA$,  on the letters $1,\ldots, n$ and $\aA_n^*$ the corresponding free monoid.

In general, a \emph{word} over an arbitrary alphabet is an element of the free monoid over that alphabet with the symbol $\varepsilon$ denoting the empty word. So, if $w=w_1\cdots w_m$ is an arbitrary word of $\aA^*$, with $w_1,\ldots,w_m\in \aA$, there are several notions that are directly related with the definition of word, which we define as follows:
\begin{itemize}
\item the \emph{length of $w$}, denoted $|w|$, is the number of symbols from $\aA$ in $w$, counting repetitions;

\item for any word $u=u_1\cdots u_k\in \aA^*$, with $u_1,\ldots,u_k\in \aA$, $u$ is a \emph{subsequence of} $w$ if there exists a sequence of indexes, $i_1,\ldots, i_k\in \mathbb{N}$, with $1\leq i_1<\ldots<i_k\leq m$, such that $u=w_{i_1}\cdots w_{i_k}$;

\item for any word $u\in \aA^*$, the word $u$ is a \emph{factor of} $w$ if there exist words $v_1,v_2\in \aA^*$, such that $w=v_1uv_2$;

\item for any $a\in \aA$, the number of occurrences of $a$ in $w$, is denoted by $|w|_a$;

\item the \emph{evaluation of} $w$, $ \ev{w}$, is the infinite sequence
of non-negative integers $(|w|_1,|w|_2,|w|_3, \ldots)$ whose $i$-th component
is $|w|_i$.
\end{itemize}
Another related concept that we will frequently use is of standard word. A \emph{standard word}  over a well-ordered  alphabet $\Sigma$ is a word $u$ that contains each symbol from the first $|u|$ symbols of $\Sigma$ exactly once. A standard word $u=u_1\cdots u_k$ in naturally identified with the permutation that maps $i\mapsto u_i$.

Next, we define two different processes for standardizing a word, both allowing us to associate to any given word, a standard word of the same length. Consider the alphabet $\mathcal{C}=\{a_{b}: a,b\in \aA\}$ well-ordered in the following way: for any $a_b,c_d\in \mathcal{C}$,
\begin{equation*}
	a_{b}\prec c_{d} \Leftrightarrow a<c \vee \left(a=c\wedge b<d\right).
\end{equation*}
 Given a symbol $a_{b}\in\mathcal{C}$, $a$ will be called the \emph{underlying symbol of} $a_{b}$ and $b$ the \emph{index of} $a_{b}$.

For any $w\in \aA^*$, the \emph{left to right standardization of $w$}, denoted $\stdl(w)$, which corresponds to the standardization presented in \cite{MR2142078}, is the word in $\mathcal{C}$ obtained by iteratively scanning all the occurrences of $1$ and labelling them with $1_1,1_2,\ldots$ from left to right, then scanning all the occurrences of $2$ and labelling them with $2_1,2_2,\ldots$ from left to right, and repeating this process for all the symbols occurring in $w$.

The \emph{right to left standardization of $w$}, denoted $\stdr(w)$, is obtained in the same way, but instead of the labelling being done from left to right, it is done from right to left.

For example, considering the word $w=1321221\in \aA_3^*$, we get
\begin{alignat*}{9}
&w && = &&1 &&{\color{blue}3}\ &&{\color{red}2}\ &&1\ &&{\color{red}2}\
&&{\color{red}2}\ &&1\ \\
&\stdl(w) && =\quad &&1_1\ &&{\color{blue}3_1}\ &&{\color{red}2_1}\ &&1_2\
&&{\color{red}2_2}\ &&{\color{red}2_3}\ &&1_3\ \\
&\stdr(w) && =\quad &&1_3\ &&{\color{blue}3_1}\ &&{\color{red}2_3}\ &&1_2\
&&{\color{red}2_2}\ &&{\color{red}2_1}\ &&1_1\ .
\end{alignat*}
Any word in $\mathcal{C}$ arising from the processes of left-to-right or right-to-left standardization will be called a standardized word.

On the other hand, for any word $w\in \mathcal{C}^*$, the \emph{de-standardization of} $w$, denoted by $\dstd(w)$, is
obtained by erasing the indexes of the underlying symbols of $w$.  Considering the previous example, where $w=1321221$,
we obtain
\begin{align*}
	\dstd(\stdl(w))=\dstd(1_13_12_11_22_22_31_3)=w
\end{align*}
and
\begin{align*}
	\dstd(\stdr(w))=\dstd(1_33_12_31_22_22_11_1)=w.
\end{align*}
As standardization never `forgets' the original symbols of the considered word, just adds indexes to them and the de-standardization just erases them, we can conclude that, for any word $w\in \aA^*$
\begin{align*}
	\dstd(\std(w))=w,
\end{align*}
where $\std$ can be replaced by either $\stdl$ or $\stdr$.

\subsection{Presentations and rewriting systems}

On a different direction, a \emph{monoid presentation} is a pair $(\Sigma,\mathcal{R})$, where $\Sigma$ is an alphabet and $\mathcal{R}\subseteq \Sigma^*\times \Sigma^*$. A monoid $M$ is said to be defined by a \emph{presentation} $(\Sigma,\mathcal{R})$ if $M\simeq \Sigma^*\!/\mathcal{R}^\#$, where $\mathcal{R}^\#$ is the smallest congruence containing $\mathcal{R}$ (see \cite[Proposition~1.5.9]{howie1995fundamentals} for a combinatorial description of the smallest congruence containing a relation). The presentation is called \emph{finite} if both $\Sigma$ and $\mathcal{R}$ are finite and \emph{multihomogeneous} if, for each symbol $a\in \Sigma$ and every defining relation $(w,w')\in\mathcal{R}$, $|w|_a=|w'|_a$. A monoid $M$ is said to be multihomogeneous if there exists a multihomogeneous presentation defining $M$.

The set of relations $\mathcal{R}$ can also be referred to as a \emph{rewriting system} whose elements are the \emph{rewriting rules}. This way, the rewriting rules $r\in\mathcal{R}$ will be written in the form $r=(r^+,r^-)$.

From this point of view, we can define another binary relation in $\Sigma^*$, $\rightarrow_\mathcal{R}$, in the following way
\begin{equation*}
	u \rightarrow_\mathcal{R} v\ \Leftrightarrow\ u=w_1r^+w_2\ \text{ and }\ v=w_1r^-w_2
\end{equation*}
for some $(r^+,r^-)\in\mathcal{R}$ and $w_1,w_2\in\Sigma^*$. The binary relation $\rightarrow_\mathcal{R}$ is known  as \emph{single-step reduction}. A word $u\in \Sigma^*$ is said to be \emph{irreducible} if there is no word $v\in \Sigma^*$ such that $u\rightarrow_\mathcal{R} v$. Define $\rightarrow_\mathcal{R}^*$ to be the transitive and reflexive closure of $\rightarrow_\mathcal{R}$. Note that the relation $\mathcal{R}^\#$ coincides with the reflexive, symmetric, and transitive closure of $\rightarrow_{\mathcal{R}}$.

We say that the rewriting system $\mathcal{R}$ is \emph{noetherian} if there is no infinite chain of single step reductions
\begin{equation*}
	w_1 \rightarrow_\mathcal{R} w_2 \rightarrow_\mathcal{R} w_3 \rightarrow_\mathcal{R} \cdots .
\end{equation*}

The rewriting system $\mathcal{R}$ is said to be \emph{confluent} if whenever $u\rightarrow_\mathcal{R}^*v$ and $u\rightarrow_\mathcal{R}^*v'$, there exists $w\in \Sigma^*$ such that $v\rightarrow_\mathcal{R}^*w$ and $v'\rightarrow_\mathcal{R}^* w$.

In the case of $\mathcal{R}$ being both noetherian and confluent, we say that $\mathcal{R}$ is \emph{complete}.

A presentation is said to be noetherian, confluent or complete if the corresponding rewriting system is, respectively,
noetherian, confluent or complete.

\section{\ps\ tableaux, insertion and \ps\ monoids}
\label{sec:pstableaux}

One of the possible ways of defining the plactic monoid is by considering it as the quotient of a free monoid by the
congruence that relates words that yield the same semistandard Young tableau, under Schensted's insertion algorithm
\cite{MR0121305}.

In this section we shall present the concepts that parallel the notions of semistandard Young tableau and Schensted's
algorithm for words in the case of Patience Sorting. The Patience Sorting algorithm was originally defined to sort a
permutation into piles.  Two natural generalizations to words (and not only for `permutations') of the Patience Sorting
algorithm appear in \cite[Subsection~2.4]{MR1694204}. Each word gives rise to a \ps-tableau, and the set of pairs of
words yielding the same \ps-tableau is a congruence. This will allow us to define the \lps\ and the \rps\ monoids in a
similar way to the plactic monoid.

\subsection{\ps -tableaux and insertion}
\label{section1.1}
A \emph{composition diagram}  is a finite collection of boxes
arranged in bottom-justified columns, where no
order on the length of the columns is imposed. The shape of a composition
diagram $B$, denoted by $\Sh({B})$, is the
sequence of the heights of the columns of $B$ from left to
right. Note that this notion of shape is  dual to the usual notion
of shape, as given in \cite{MR1905123}.

An \emph{\lps\ tableau} (respectively, an \emph{\rps\ tableau}) is a
composition diagram  with entries
from $\aA$, so that the sequence of entries of the boxes in each column is
strictly (resp., weakly) decreasing  from top to bottom, and the
sequence of entries
of the boxes in the bottom row is weakly (resp., strictly) increasing
from left to right.

Consider the following diagrams:
	\begin{equation} \label{exmp1.1}
	\ytableausetup
	{centertableaux=bottom}
	B = \begin{ytableau}
		\none & 5 & 6 & \none \\
		\none & 4 & 3 & 5 \\
		2 & 2 & 2 & 4
	\end{ytableau}\qquad
	C = \begin{ytableau}
	2 & 5 & 6 \\
	2 & 4 & 5 \\
	2 & 3 & 4
	\end{ytableau}
	\end{equation}
$B$ is an \lps\ tableau with shape $\Sh({B})=(1,3,3,2)$, whereas $C$ is an \rps\ tableau with shape $\Sh({C})=(3,3,3)$.

Henceforth, we shall often refer to an \lps\ tableau or to an \rps\ tableau simply as a \ps\ tableau, not distinguishing
the cases when they can be dealt in a similar way.


The Patience Sorting algorithm 
defined only for permutations by Burstein and Lankham \cite[Algorithm 1.1]{BL2005}, can be generalized to arbitrary words in two different ways, depending on whether or not we allow repeated symbols on the same column \cite[Subsection~3.2]{THOMAS2011610}.
We begin by presenting the algorithm that computes
from an \lps\  (resp., \rps)
tableau $B$ and a symbol $a\in \aA$,   an \lps\ (resp., \rps) tableau $B \leftarrow a$.

\begin{algorithm}[Right insertion of a symbol on a \ps\ tableau]\label{alg:Bell}
~\par\nobreak
\textit{Input:} An \lps\ (resp., \rps) tableau $B$ and a symbol $a\in \aA$.

\textit{Output:} An \lps\  (resp., \rps) tableau $B \leftarrow a$.

\textit{Method:}
\begin{enumerate}
 \item If  $a$ is greater than or equal to (resp., greater
than) every entry in the bottom row of
$B$, create a box with $a$ as an entry at the rightmost end of the bottom row
of $B$, and output the resulting tableau. (Note that the insertion can be done
to
an empty PS tableau, $\emptyset$, which corresponds to the creation of a PS tableau with a single box).
\item Otherwise, let $z$ be the leftmost  letter in the bottom row of $B$ that
is greater (resp., greater than or equal) than $a$. Create a box at the
top of the column tableau where
$z$ was placed, move  each element of the column to the box strictly above, and let $a$ be the
symbol on the bottom box. Output the resulting tableau.
\end{enumerate}
\end{algorithm}

It is straightforward to check that the procedure outputs an \lps\  (resp., \rps)
tableau whenever the algorithm starts with an \lps\ (resp., \rps) tableau.

If in the previous algorithm, the symbol to be inserted and the symbols in the
\ps\ tableaux are all distinct, then Algorithm~\ref{alg:Bell} follows the same
procedure as Algorithm 1.1 in \cite{BL2005}. Moreover, both the \lps\ version and
the \rps\ version execute the algorithm in the same way.

The iterated insertion of the symbols of a word (read from left to right)
into the resulting tableaux using the previous algorithm gives rise to the
\ps\ insertion of a word, which is described in detail by the following
algorithm:

\begin{algorithm}[\ps\ algorithm for words]\label{alg:Bellword}
	~\par\nobreak
	\textit{Input:} A word $w=w_1\cdots w_k\in \aA^*$, with $w_1,\ldots,w_k\in
	\aA$.

	\textit{Output:} An \lps\ (resp., \rps) tableau $\bB_\ell(w)$ (resp., $\bB_r(w)$).

	\textit{Method:}
Start with the empty tableau $P_0=\emptyset$. For each $i=1,\ldots, k$, insert $w_i$ into the \lps\ (resp., \rps) tableau $P_{i-1}$ using Algorithm~\ref{alg:Bell}. Output $P_k$ for $\bB_\ell(w)$ (resp., $\bB_r(w)$).
\end{algorithm}

Note that although Algorithms~\ref{alg:Bell} and \ref{alg:Bellword} consider words on the alphabet $\aA$, they also can be used considering words of any totally ordered alphabet. Hence, in any of the results of this paper, the alphabet $\aA$ can be replaced by some other totally ordered alphabet.

We use the notation $\bB(w)$ to refer to
any of the tableaux $\bB_\ell(w)$ or $\bB_r(w)$, indistinguishably.
Denoting by $\emptyset$ the empty \ps\ tableau, if $w=w_1\cdots w_k$, we have
\begin{equation*}
	\bB(w) = \Big(\cdots \big((\emptyset \leftarrow w_1) \leftarrow w_2 \big)\cdots
	\Big)\leftarrow w_k.
\end{equation*}

For example, if we are given the word $w=254263542$ its \lps\ tableau is
obtained as pictured in
Figure~\ref{figure:right_insertion}.
\begin{figure}[ht]
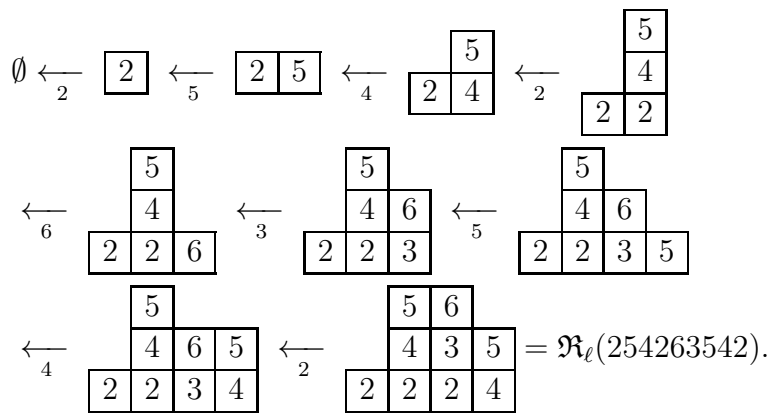
\label{figure:right_insertion}
\centering
\begin{align*}
&\emptyset \xleftarrow[2]{\ \ \ }\
\ytableausetup
{mathmode, boxsize=1.3em, aligntableaux=center}
\begin{ytableau}
2
\end{ytableau}\
\xleftarrow[5]{\ \ \ }\
\begin{ytableau}
2 & 5
\end{ytableau}\
\xleftarrow[4]{\ \ \ }\
\begin{ytableau}
\none & 5\\
2 & 4
\end{ytableau}\
\xleftarrow[2]{\ \ \ }\
\begin{ytableau}
\none & 5\\
\none & 4\\
2 & 2
\end{ytableau}\\
&\xleftarrow[6]{\ \ \ }\
\begin{ytableau}
\none & 5 & \none\\
\none & 4 & \none\\
2 & 2 & 6
\end{ytableau}\
\xleftarrow[3]{\ \ \ }\
\begin{ytableau}
\none & 5 & \none\\
\none & 4 & 6\\
2 & 2 & 3
\end{ytableau}\
\xleftarrow[5]{\ \ \ }\
\begin{ytableau}
\none & 5 & \none\\
\none & 4 & 6\\
2 & 2 & 3 & 5
\end{ytableau}\\
&\xleftarrow[4]{\ \ \ }\
\begin{ytableau}
\none & 5 & \none & \none \\
\none & 4 & 6 & 5 \\
2 & 2 & 3 & 4
\end{ytableau}\
\xleftarrow[2]{\ \ \ }\
\begin{ytableau}
\none & 5 & 6 & \none \\
\none & 4 & 3 & 5 \\
2 & 2 & 2 & 4
\end{ytableau}=\bB_\ell(254263542).\
\end{align*}
\caption{Execution of Algorithm~\ref{alg:Bellword}  to the word $w=254263542$, where below each arrow we indicate
the symbol to be inserted next, and on its right the result of that insertion
into the previous tableau.}
\end{figure}
As noticed, if all symbols of a word are distinct the resulting tableau is the
same, independently of  using the \lps\ or the \rps\ version of
Al\-gorithm~\ref{alg:Bellword} or \cite[Algorithm~1.1]{BL2005}. So, if $\sigma$
is a word where the symbols are all distinct, our notation $\bB(\sigma)$ agrees
with the notation used in \cite{BL2005} to denote the output of a permutation
using Algorithm~1.1, and we have
\begin{equation}\label{exmp1.2}
\bB_\ell(\sigma)=\bB(\sigma)=\bB_r(\sigma)
\end{equation}
for any such $\sigma$.


The following lemma relates the position of repeated symbols in the word to their position in the tableau.

\begin{lem}\label{lemma:equal_symbols_in_tableau}
Let $w$ be a word on $\aA$ and let  $a_1,a_2\in\aA$ be two symbols of $w$ with $a_1\leq a_2$. Then $a_1a_2$ is a subsequence of $w$ if, and only if, the symbol $a_2$ is positioned in $\bB(w)$ either in the same column of $a_1$, below $a_1$, (only in the \rps\ case) or  in a column further to the right of the column containing $a_1$.
\end{lem}

\begin{proof}
Suppose that $a_1a_2$ is  a subsequence of $w$,  that is, $w=w_1a_1w_2a_2w_3$. In the computation of $\bB_\ell(w)$ using Algorithm~\ref{alg:Bellword}, the symbol $a_2$ is inserted after $a_1$, and thus will be in a column further to the right of the column in which $a_1$ is positioned. Indeed, on the step where the symbol $a_2$ is inserted, the bottom symbol of the column in which $a_1$ is positioned  is smaller than $a_1$, and thus by   Algorithm~\ref{alg:Bell}, the symbol $a_2$ will be inserted to a column further to the right.

 In $\bB_r(w)$ the reasoning is similar, with the additional condition that $a_2$ can be inserted in the same column below $a_1$ if during the execution of Algorithm~\ref{alg:Bell}, the bottom  symbol currently in the column of $a_1$ is equal to $a_1$.

Conversely,  suppose that $a_1a_2$ is not a subsequence of $w$. This means that $a_2a_1$ is a subsequence of $w$ and therefore, by the direct part of the statement already proved, applied to the subsequence $a_2a_1$, we reach a contradiction.
\end{proof}

Two words $u,v$, possibly over different totally ordered alphabets, whose \lps\ tableaux have
the same shape (and thus have the same length) are said to have
\emph{equivalent $l$-insertions} 
if for every
$i=1,\ldots, |u|$ the $i$-th symbol of $u$ and the $i$-th symbol of $v$ are in
the same (column-row) position of the respective tableaux
$\bB_\ell(u)$ and $\bB_\ell(v)$. Similarly, we define words to have \emph{equivalent
$r$-insertions} by replacing $\bB_\ell(u)$ and $\bB_\ell(v)$ by $\bB_r(u)$ and
$\bB_r(v)$, respectively, in the previous definition.

The words $u=24131$ and $v=36142$, whose \lps\ tableaux are
	\begin{equation*}
	\ytableausetup{mathmode, boxsize=1.3em, aligntableaux=center}
	\bB_\ell(u)=
	\begin{ytableau}
	\none & 4\\
	2 & 3\\
	1 & 1
	\end{ytableau}\qquad \text{and}\qquad \bB_\ell(v)=
	\begin{ytableau}
	\none & 6\\
	3 & 4\\
	1 & 2
	\end{ytableau}
	\end{equation*}
have equivalent $l$-insertions.
%
%
%
%

\begin{lem}\label{lemma:equivalent_insertions}
	For any $w\in \aA^*$, the words $w$ and $\stdl(w)$ have equivalent
$l$-insertions, and $w$ and $\stdr(w)$ have equivalent $r$-insertions.
\end{lem}
\begin{proof}
	Let $w\in \aA^*$. The proof follows by induction on the length of $w$,
$|w|$. The result holds trivially for  $|w|=1$.

	Suppose by induction hypothesis that the result holds for any word $w$
with $|w|=n>1$. Let $a_1,a_2,\ldots,a_k$ and $a'_1,a'_2,\ldots,a'_k$ be
the sequence of entries in the boxes of the bottom row of, respectively
$\bB_\ell(w)$ and $\bB_\ell\big(\stdl(w)\big)$ (resp., $\bB_r(w)$ and $\bB_r\big(\stdr(w)\big)$).
Consider $a\in \aA$ and let $a'$ be the rightmost element in the word
$\stdl(wa)$ (resp., $\stdr(wa)$).

If $a_k\leq a$ (resp., $a_k< a$), then we
are in case (1) of  Algorithm~\ref{alg:Bell} and therefore a box is added
to the rightmost end of the bottom row of $\bB_\ell(w)$ (resp., $\bB_r(w)$), with
the symbol $a$ in it.

The symbol $a'_k$ is to the left of the symbol $a'$ in the word $\stdl(wa)$
(resp., $\stdr(wa)$). By the left-to-right standardization (resp., right-to-left
standardization) we have $a'_k<a'$, since $a_k\leq a$ (resp., $a_k< a$). To
compute $\bB_\ell\big(\stdl(wa)\big)=\bB_\ell\big(\stdl(w)\big)\leftarrow a'$ (resp.
$\bB_r\big(\stdr(wa)\big)=\bB_r\big(\stdr(w)\big)\leftarrow a'$) we  also use step
 (1) of Algorithm~\ref{alg:Bell}. A box is added
to the rightmost end of the bottom row of $\bB_\ell\big(\stdl(w)\big)$ (resp.
$\bB_r\big(\stdr(w)\big)$), with the symbol $a'$ in it. Therefore,
the $n+1$-th symbols of words $wa$ and $\stdl(wa)$ (resp., $\stdr(wa)$)
take the same position of the respective tableaux. Thus, $wa$ and $\stdl(wa)$
(resp., $wa$ and $\stdr(wa)$) have equivalent $l$-insertions (resp.
$r$-insertions).

Suppose now that $a< a_k$ (resp., $a\leq a_k$). Let $a_j$, with $j\in
\{1,2,\ldots, k\}$ be the leftmost letter in the bottom row of $\bB_\ell(w)$ (resp.
$\bB_r(w)$) such that $a<a_j$ (resp., $a\leq a_j$). We have the following
situation:
	\begin{equation*}
	\label{order1}
	a_1\leq \ldots\leq a_{j-1}\leq a<a_j\leq \ldots\leq a_k
	\end{equation*}
	\begin{equation*}
	\label{order2}
	(\text{resp., } a_1< \ldots< a_{j-1}< a\leq a_j< \ldots< a_k.)
	\end{equation*}
	Therefore, case (2) of Algorithm~\ref{alg:Bell} is applied
and so $a$ bumps all the elements of the column of $a_j$, being inserted below
the $a_j$ in the bottom row.

By the left-to-right standardization (resp., right-to-left
standardization) the
inequalities  $a'_{j-1}< a'$ and $a'<a'_j$ hold, since $a_{j-1}\leq a$ and
$a<a_j$ (resp., $a_{j-1}< a$ and $a\leq a_j$). Therefore, to compute
$\bB_\ell(\stdl(w))\leftarrow a $ (resp., $\bB_r(\stdr(w))\leftarrow a $)  we apply case (2)
of Algorithm~\ref{alg:Bell}, and so $a'$ bumps all the elements of the
column of $a'_j$ being inserted in the bottom row below $a'_j$.

Clearly, the $n+1$-th symbols of the words $wa$ and $\stdl(wa)$ (resp.
$\stdr(wa)$)
take the same position of the respective tableaux. Similarly, the symbols of
the columns $a_j$ and $a'_j$ will have the same relative positions in the new
tableaux, by the induction hypothesis.
Therefore, $wa$ and $\stdl(wa)$ (resp., $wa$ and $\stdr(wa)$) have equivalent
$l$-insertions (resp., $r$-insertions). The result follows by induction.
\end{proof}

Similarly to the standardization and de-standardization of words presented in Subsection~\ref{alphabetswords}, we present here corresponding standardization and de-standardization processes of \ps\ tableaux.

Given a \ps\ tableau $B$, let $\Stdl(B)$ be the tableaux obtained from $B$ reading the entries of $B$ column by column, from left to right, and on each column from top to bottom, and attaching to each symbol $a\in\aA$ an index $h$ to the $h$-th appearance of $a$.  Similarly, $\Stdr(B)$ is the tableaux obtained from $B$ reading the entries of $B$ column by column, from right to left, and on each column from bottom to top, attaching to each symbol $a\in\aA$ an index $h$ to the $h$-th appearance of $a$.

 If the \ps\ tableaux $B$ has symbols form a standardized word, the de-standardization of $B$, denoted $\Dstd(B)$, is the tableau produced by erasing the indexes of each of its underlying symbols.

The following
result relates both versions of Algorithm~\ref{alg:Bellword}, with Algorithm~1.1
 of \cite{BL2005}.

\begin{prop}\label{obs1.4}
For any word $w\in \aA^*$ and $x \in \{\ell,r\}$ we have
\begin{enumerate}
\item $\bB_x(\std_x(w))=\Std_x(\bB_x(w))$; and
\item $\Dstd\big(\bB_x\big(\std_x(w)\big)\big)=\bB_x(w)$.
\end{enumerate}
\end{prop}

\begin{proof}
Let $w\in \aA^*$. By Lemma~\ref{lemma:equivalent_insertions}, the words $\stdl(w)$ and $w$ have equivalent $l$-insertions, that is, the $i$-th symbol of $\stdl(w)$ and the $i$-th symbol of $w$ are, respectively, in the same position of $\bB_\ell(\stdl(w))$ and of $\bB_\ell(w)$. Since
 elements in corresponding positions of $w$ and
$\stdl(w)$ have the same underlying symbol, and $\Dstd$
just erases the indexes of the symbols in the tableau, we deduce part (2) of the statement.

It remains to show that during the process of standardization of $\bB_\ell(w)$ the indexes will match, so that the symbols in corresponding positions of $\bB_\ell(\stdl(w))$ and $\Stdl(\bB_\ell(w))$ are equal.

Let $a\in\aA$ be a symbol in $w$ and let $k=|w|_a$. So $a_1\cdots a_k$ is a subsequence of $\stdl(w)$. By Lemma~\ref{lemma:equal_symbols_in_tableau} equal symbols in $w$, as we proceed from left to right in $w$, will be appearing in $\bB_\ell(w)$ as we proceed going column by column, from left to right, and on each column from top to bottom. Thus, by definition of $\Stdl$, the standardization of $\bB_\ell(w)$ will attach indexes in the same way the standardization of $w$ is processed.

The right-to-left case follows the same reasoning in the first two paragraphs of this proof. In the third paragraph, the main difference is that $a_k\cdots a_1$ is the subsequence of $\stdr(w)$ and that also Lemma~\ref{lemma:equal_symbols_in_tableau} guarantees that equal symbols in $w$, as we proceed from \emph{right to left} in $w$, will be appearing in $\bB_r(w)$ as we proceed going column by column, from \emph{right to left}, and on each column from  \emph{bottom to top}. This agrees with the standardization $\Stdr$ of $\bB_r(w)$.
\end{proof}

\subsection{From \ps\ tableaux to words}
\label{subsect1.2}

All the concepts introduced in the following paragraphs have an \lps\ and an \rps\ version depending on whether the considered \ps\ tableau is, respectively, \lps\ or \rps. So, for the sake of brevity we will only introduce the general notions and when necessary we distinguish them.

As we will see, it will be convenient to pass from diagrams to words
representing such diagrams. The first concept that we will need is the
generalization of reverse patience word presented in \cite{BL2005}. So, given a
\ps\ tableau $B$, the \emph{column reading of} $B$, denoted by $\cC(B)$, is the
word obtained from reading the entries of the \ps\ tableau $B$, column by
column, from the leftmost to the rightmost, starting on the top of  each column
and ending on its bottom. For example, the column reading of the \lps\ tableau
$B$ in (\ref{exmp1.1}) is $\cC(B)=2\,542\,632\,54$, whereas the column reading
of the \rps\ tableau $C$ in (\ref{exmp1.1}) is $\cC(C)=222\,543\,654$.

The following equality  follows from the definition of standardization of words and  tableaux: for any $x \in \{\ell,r\}$ and any $x$\ps\ tableau $B$,
\begin{equation}\label{eq:column_reading_and_standardization_commute}
\cC(\Std_x(B))  =   \std_x(\cC(B)).
\end{equation}

The relations established between Algorithm~\ref{alg:Bellword} for words, Algorithm 1.1 of \cite{BL2005} for permutations and the standardization processes, allows us to transfer results from the permutation case, such as
 \cite[Lemma~2.4]{BL2005},  to the general case of words.
\begin{lem}
	\label{propo1.5}
	For any $x \in \{\ell,r\}$ and word $w\in \aA^*$,
	\begin{equation*}
	\bB_x\Big(\cC\big(\bB_x(w)\big)\Big)=\bB_x(w).
	\end{equation*}
\end{lem}
\begin{proof}
	It follows by Proposition~\ref{obs1.4}{(2)} that
	\[
          \bB_x\Big(\cC\big(\bB_x(w)\big)\Big) =\Dstd_x\Big(\bB_x\Big(\std_x\big(\cC\big(\bB_x(w)\big)\big)\Big)\Big).
        \]
By (\ref{eq:column_reading_and_standardization_commute}) and Proposition~\ref{obs1.4}{(1)} we get
\begin{align*}
\bB_x\Big(\cC\big(\bB_x(w)\big)\Big) & =\Dstd_x\Big(\bB_x\Big(\cC\big(\Std_x\big(\bB_x(w)\big)\big)\Big)\Big)\\
 & =\Dstd_x\Big(\bB_x\Big(\cC\big(\bB_x\big(\std_x(w)\big)\big)\Big)\Big).
\end{align*}

Now, by \cite[Lemma~2.4]{BL2005} the result holds for standard words, so
\[
  \bB_x\Big(\cC\big(\bB_x(w)\big)\Big) = \Dstd\Big(\bB_x\big(\std_x(w)\big)\Big).
  \]
	The result now follows by Proposition~\ref{obs1.4}{(2)}.
\end{proof}

Let $x \in \{\ell,r\}$. A word on $\aA^*$ that is the reading of a column of some $x$\ps\ tableau is called an $x$\ps\
\emph{co\-lumn word}. (In \cite{BL2005}, the term `pile' is used.) Similarly, a word on $\aA^*$ that is the reading of
the bottom row of some $x$\ps\ tableau is called a \ps\ \emph{bottom row word}.  We call a word an $x$\ps\
\emph{canonical word} if it is the column reading of some $x$\ps\ tableau. Note that \rps\ column words and \lps\ bottom
row words are weakly increasing, while \lps\ column words and \rps\ bottom row words are strictly increasing.

Notice that any $x$\ps\ canonical word $c$ has a unique decomposition into $x$\ps\ column words $c_1,\ldots,c_k$, such
that $c=c_1\cdots c_k$ and such that the subsequence obtained by taking the smallest symbols of each $c_i$ is a \ps\
bottom row word.

\begin{prop} \label{prop2.3}
The mapping $\bB_x$  from the set of canonical words to the set of \ps\  tableaux is a bijection, and $\cC$ is its inverse map.
\end{prop}

\begin{proof}
  Let $c$ be a $x$\ps\ canonical word and let $c = c_1\cdots c_k$ be its column decomposition. Proceed with
  Algorithm~\ref{alg:Bellword} to compute $\bB_x(c)$. Note that $\bB_x(c_i)$ is a single column, for any
  $i = 1,\ldots,k$. Since the subsequence composed by the smallest symbols of the $c_i$ is a bottom row word, every symbol
  in $c_i$ is, in the \lps\ case, greater than or equal to, or, in the \rps\ case, greater than every symbol in
  $\bB_x(c_1\cdots c_{i-1})$, and so is inserted into a new column to the right of $\bB_x(c_1\cdots c_{i-1})$. Hence
  $\bB_x(c)$ is the juxtaposition of $\bB_x(c_1),\ldots,\bB_x(c_k)$. Thus $\bB_x$ is a bijective map on canonical words
  and $\cC\bigl(\bB_x(c)\bigr) = c$.
\end{proof}

As a consequence of the previous result, we can identify \ps\ tableaux with their canonical words. Throughout the
remainder of this paper we shall use this identification.

From any word $w\in \aA^*$ we can produce a canonical word $w_c$ by computing the tableau $\bB_x(w)$ and identifying it with the column reading $\cC(\bB_x(w))$. Let  $c_1,\ldots,c_k$ be the unique decomposition of $w_c$ into $x$\ps\ column words, such that $c=c_1\cdots c_k$ and   the subsequence obtained by taking the
smallest  symbols of each $x$\ps\ column word, is an $x$\ps\ bottom row word.
 We refer to the sequence of column
words  $(c_1,\ldots,c_k)$ as
the $x$\ps\  \emph{column configuration of} $w$. (In \cite{BL2005}, the term `pile configuration' is used.)

By Proposition~\ref{prop2.3} and Lemma~\ref{propo1.5} the mapping from $\aA^*$ to the set of $x$\ps\ canonical words given by $w\mapsto\cC(\bB_x(w))$, for any $w\in\aA^*$, is a surjection whose restriction to the set of canonical words is bijective.



%
%
%
%

\subsection{The left-to-right minimal
subsequence}
In this subsection we present the construction of sequences of subsequences of a
given a word, that are built going through the word from left to right.
These sequences are the generalization of the computation of the left-to-right minimal
subsequence for permutations in \cite[Definition~2.8]{BL2005}. We shall see
that they provide us an alternative way to compute directly the column
configuration of a word. This new approach will allow us to define \lps\ and \rps\ left insertion
algorithms on next subsection.

The algorithm is based on the computation of a strictly (resp., weakly)
decreasing subsequence from a given word $w\in \aA^*$, called its
\emph{$l$-decreasing} (resp., \emph{$r$-decreasing}) \emph{subsequence} and denoted
$\ldd(w)$ (resp., $\rdd$):
\begin{algorithm} (Decreasing subsequence)\label{alg:decreasing}
~\par\nobreak
\textit{Input:} A word $w\in \aA^*$.

\textit{Output:} Its $l$-decreasing ($r$-decreasing) subsequence $\ldd(w)$ (resp.
$\rdd$).

\textit{Method:}
\begin{enumerate}
	\item The first symbol of $\ldd(w)$ (resp., $\rdd(w)$) is the first symbol of $w$;
	\item The $i+1$-th symbol of $\ldd(w)$ (resp., $\rdd(w)$) is obtained from the $i$-th
symbol of $\ldd(w)$ (resp., $\rdd$), by going to the position that the $i$-th
symbol of $\ldd(w)$ (resp., $\rdd$) occupies in $w$ and finding the first symbol
to its right that is less than (resp., less than or equal to) it. It stops when no
symbol that is less than (resp., less than or equal to) the $i$-symbol is found and outputs the obtained
$l$-decreasing (resp., $r$-decreasing) subsequence.
\end{enumerate}
\end{algorithm}
The algorithm of left-to-right minimal subsequences is then defined in the following way:
\begin{algorithm} (Left-to-right minimal subsequence)\label{alg:lefttoright}
	~\par\nobreak
	\textit{Input:} A word $w\in \aA^*$.

	\textit{Output:} Its $l$-left-to-right ($r$-left-to-right) minimal
subsequence $\ldD(w)$ (resp., $\rdD(w)$).

	\textit{Method:}
	\begin{enumerate}
	\item Let $\ldd_1(w)$ (resp., $\rdd_1(w)$) be the left-to-right
minimal subsequence $\ldd(w)$ (resp., $\rdd(w)$).
\item For $i \geq 2$, let $w^\ell_{(i)}$ (resp., $w^r_{(i)}$) be the word obtained from $w$ by deleting the symbols
  appearing in $\ldd_1(w),\ldots,\ldd_{i-1}(w)$ (resp., $\rdd_1(w),\ldots,\rdd_{i-1}(w)$). Define $\ldd_i(w)$ (resp., $\rdd_i(w)$) to be $\ldd(w^\ell_{(i)})$ (resp., $\rdd(w^r_{(i)})$).
\item Output the sequence of non-empty words $\big(\ldd_1(w),\ldots,
\ldd_i(w)\big)$ (resp., $\big(\rdd_1(w),\ldots, \rdd_i(w)\big)$).
	\end{enumerate}
\end{algorithm}
For any $i\in \mathbb{N}$, the word $\ldd_i(w)$ (resp., $\rdd_i(w)$) computed by
the previous algorithm is defined as the $i$-th $l$-\emph{decreasing} (resp.
$r$-\emph{decreasing}) \emph{subsequence of} $w$, which is the generalization of
the $i$-th left-to-right minimal subsequence for permutations in
\cite{BL2005}.

So, for instance, considering the word $w=256423542$, we get:
\begin{alignat*}{2}
& \ldd_1(w)=\, \ldd(256423542)=2\ &&\ \rdd_1(w)=\ \rdd(256423542)=222\\
& \ldd_2(w)=\, \ldd(\cancel{2}56423542)=542\ &&\ \rdd_2(w)=\, \rdd(\cancel{2}564\cancel{2}354\cancel{2})=543\\
& \ldd_3(w)=\, \ldd(\cancel{2}\cancel{5}6\cancel{4}\cancel{2}3542)=632\ &&\ \rdd_3(w)=\, \rdd(\cancel{2}\cancel{5}6\cancel{4}\cancel{2}\cancel{3}54\cancel{2})=654\\
& \ldd_4(w)=\, \ldd(\cancel{2}\cancel{5}\cancel{6}\cancel{4}\cancel{2}\cancel{3}
54\cancel { 2 })=54 &&\ \rdd_4(w)=\, \rdd(\cancel{2}\cancel{5}\cancel{6}\cancel{4}\cancel{2}\cancel{3}\cancel{5}\cancel{4}\cancel{2}) =\varepsilon\\
&\ldd_5(w)=\, \ldd(\cancel{2}\cancel{5}6\cancel{4}\cancel{2}\cancel{3}\cancel{5}\cancel{4}\cancel{2})=\varepsilon
\end{alignat*}
therefore $\ldD(w)=(2,542,632,54)$ and $\rdD(w)=(222,543,654)$.

Whenever possible, we will use $\dD$ instead of $\ldD$ (resp., $\rdD$) and $\dd$ instead of $\ldd$ (resp., $\rdd$), in order to simplify the notation.


Similarly to the notion of equivalent $l$-insertions
previously defined, we say that two words $u,v$, possibly over different totally  ordered
alphabets, have \emph{equivalent $l$-left-to-right minimal subsequences}  if:
\begin{itemize}
\item both words have the same number $k$ of non-empty $l$-decreasing  subsequences; and,
\item for every $i=1,\ldots,k$, the positions of the symbols of the subsequence $\ldd_i(u)$ in $u$ are the same as the positions of the symbols of the subsequence $\ldd_i(v)$ in $v$.
\end{itemize}
The notion of \emph{equivalent $r$-left-to-right minimal subsequences} is defined similarly.

For example, the words $u=24131$ and $v=36142$ have equivalent $l$-left-to-right minimal subsequences, since $\ldD(u)=(21,431)$ and $\ldD(v)=(31,642)$.

\begin{lem}
For any word $w\in \aA^*$, the words $w$ and $\stdl(w)$ have equivalent
$l$--left-to-right minimal subsequences, and $w$ and $\stdr(w)$ have equivalent
$r$--left-to-right minimal subsequences.
\end{lem}
\begin{proof}
We present the full proof for the $l$-case, and whenever the reasoning for the $l$-case and the $r$-case has some significant difference we present the arguments for the $r$-case in brackets.

The proof follows by induction on the length of $w$.
If $|w|=1$, the result holds trivially. Suppose by induction
hypothesis that the result holds for a word $w$ with $|w|=n\geq 1$. Let $a\in \aA$ and let $k$ be the number of non-empty $l$-decreasing subsequences of both $w$ and $\stdl(w)$.

We will prove that $wa$ and $\stdl(wa)$ have equivalent $l$-left-to-right minimal subsequences.
Suppose $w=w_1\cdots w_n$, with $w_1,\ldots, w_n\in \aA$.
Then $\stdl(w)=w'_1\cdots w'_n$,
with $w'_1,\ldots, w'_n\in \mathcal{C}$. Denote by  $a'$  the rightmost symbol of $\stdl(wa)$.

For $i=1$, suppose that $w_j$ is the last symbol of $\ldd_i(w)$
and let $w'_j$ be the corresponding symbol in $\stdl(w)$.
As by induction hypothesis $w$ and $\stdl(w)$  have equivalent $l$-left-to-right minimal subsequences, $w'_j$ is the last symbol of $\ldd_i(\stdl(w))$.

If $w_j\leq a$ (resp., $w_j<a$) then, by the left-to-right standardization (resp., right-to-left standardization) $w'_j<a'$. Thus $\ldd_i(wa)=\, \ldd_i(w)$  and $\ldd_i\big(\stdl(wa)\big)=\,\ldd_i\big(\stdl(w)\big)$.
Increase $i$ by one, and repeat the previous argument until $w_j> a$ (resp., $w_j\geq a$). If this case never occurs, then both words $wa$ and $\stdl(wa)$ have a new $l$-decreasing subsequence  $\ldd_{k+1}(wa)=a$ and $\ldd_{k+1}(\stdl(wk))=a'$ where its symbols have the same position on $w$ and $\stdl(wa)$, respectively.

Suppose we have reached an $i\in\{1,\ldots,k\}$ such that $w_j> a$ (resp., $w_j\geq a$). Then by the left-to-right standardization (resp., right-to-left standardization) $w'_j>a'$. So, $\ldd_i(wa)=\, \ldd_i(w)a$  and $\ldd_i\big(\stdl(wa)\big)=\,\ldd_i\big(\stdl(w)\big)a'$.

After this case $ a$ and $a'$ have been deleted in the computation of the left-to-right (resp., right-to-left) minimal subsequence, and so the computation of the remaining $l$-decreasing subsequences of $\ldD(wa)$ and $\ldD(\stdl(wa))$ proceeds in the same way as in the computation of $\ldD(w)$ and $\ldD(\stdl(w))$.

As $a$ and $a'$ occupy the same positions in $wa$ and $\stdl(wa)$, respectively and by the induction hypothesis $w$ and $\stdl(w)$ have equivalent  $l$-left-to-right minimal subsequences, we can conclude that $wa$ and $\stdl(wa)$ have equivalent $l$-left-to-right minimal subsequences.

By induction the result follows.
\end{proof}

 As by the previous result, for any $w\in \aA^*$, the words $w$ and $\stdl(w)$ (resp., $\stdr(w)$) have equivalent ($r$-) $l$-left-to-right minimal subsequences and standardization just adds indexes to the symbols of $w$ and $\dstd$ just erases them, we can conclude the following:

\begin{lem}
	\label{propo3.11}
	For any $w\in \aA^*$,  the de-standardization of the left-to-right minimal subsequence of $\std(w)$ is the left-to-right minimal subsequence of $w$.
\end{lem}

Proposition~\ref{prop2.3}  shows that \ps\ Algorithm (Algorithm~\ref{alg:Bellword}) can be viewed has a
computation of the column configuration associated to a word $w\in \aA^*$. The
following result, which is the generalization of \cite[Lemma 2.9]{BL2005}, shows that the new algorithm produces the same column configuration.

\begin{prop}
For any $w\in \aA^*$, the left-to-right minimal subsequence of $w$ produces the column configuration of $w$.
\end{prop}
\begin{proof}
	As the argument applies to both $l$ and $r$-cases, we just provide the proof for the $l$-case.
	Let $w\in \aA^*$ and let $(c_1,\ldots, c_k)$ be the \lps\ column configuration of $w$. By Proposition~\ref{obs1.4} the \lps\ column configuration of $w$ can be obtained by computing the \lps\ column configuration of $std(w)$, and then de-standardizing the resulting words (removing its indexes).

	It is known \cite[Lemma~2.9]{BL2005}, that the column configuration of $std(w)$ corresponds to the left-to-right minimal subsequence of $std(w)$. Thus, the de-standardization of the left-to-right minimal subsequence of $std(w)$ is precisely $(c_1,\ldots, c_k)$. The result now follows since, by  the previous lemma,  the de-standardization of the left-to-right minimal subsequence of $std(w)$ is the left-to-right minimal subsequence of $w$.
\end{proof}

For $x \in \{\ell,r\}$, given a word $w\in \aA^*$, if $\dD_x(w)=(\dd^x_1(w),\ldots,\dd^x_k(w))$ and $(c_1,\ldots,c_j)$ is the column configuration of $w$, then by the previous proposition it follows that $k=j$ and $\dd^x_i(w)=c_i$, for all $i\in \{1,\ldots,k\}$, and thus $\cC(\bB_x(w))=\dd^x_1(w)\cdots \dd^x_k(w)$.

\subsection{Left-insertion}
\label{subsection1.3}

This section presents a new  algorithm
which performs  left insertion of symbols on \ps\ tableaux and thus allows us to build \ps\ tableaux from words, proceeding from left to right on the word,
which we will prove that has the same output as the Patience Sorting algorithm
(Algorithm~\ref{alg:Bellword}). These algorithms will be crucial in proving
biautomaticity for the \lps\ monoid in Section~\ref{sec4}.

The procedure makes use of the fact that the computation of the left-to-right
minimal subsequence produces the same result as \ps\ algorithm.

The following lemma gives us some of the tools to easily compute   decreasing subsequences, by left inserting symbols on a canonical word.

\begin{lem}\label{lem:initial_decreasing_sequence}
 Let  $c'\in \aA^*$ be a \ps\ column word, and let $w\in \aA^*$ be a canonical word
with column configuration $(c_1,\ldots, c_k)$. Then $\dd(c'w)=\dd(c'c_1)$ and
either:
 \begin{enumerate}
 \item $c'c_1$ is a \ps\ canonical word, with $\dd(c'c_1)=c'$ or $\dd(c'c_1)=c'c_1$,
and hence $c'c_1\cdots c_k$ is a \ps\ canonical word; or
 \item $c'c_1$ is not a \ps\ canonical word, and
$\dd(c'c_1)$ has the same minimum symbol as $c_1$, and the
word obtained from $c'c_1$ by erasing the symbols from $\dd(c'c_1)$
is the column word $\dd_2(c'c_1)$, that is a prefix of $c_1$.
\end{enumerate}
\end{lem}
\begin{proof}
Note that the
minimum symbol in $w$ that is further left (resp., right), is the minimum symbol of $c_1$.

Following Algorithm~\ref{alg:decreasing}, to compute $\ldd(c'w)$ (resp., $\rdd(c'w)$) we obtain first all symbols from $c'$, since $c'$ is an \lps\ (resp., \rps) column word. If
the minimum symbol in $c'$ is less than or equal (resp., less than) the
minimum symbol of $c_1$, then $\ldd(c'w)=c'=\, \ldd(c'c_1)$ (resp., $\rdd(c'w)=c'=\, \rdd(c'c_1)$).

Otherwise, the algorithm proceeds
passing to symbols from $w$, it will search for the first (left to right)
symbol in $w$ that is less than (resp., less than or equal to) the minimum symbol in $c'$. So, the algorithm will find such element in the subsequence $c_1$. Indeed, let
$c_1=a_k \cdots a_1$, with $a_i\in\aA$, and
$s\in\{1,\ldots,k\}$ be the maximum index such that the minimum of $c'$ is greater than (resp., greater than or equal to) $a_s$.  (Note that $a_1$ is the minimum symbol in $w$ further left (resp., right).)
Then computing $\ldd(c'w)$ (resp., $\rdd(c'w)$) we get $c'a_s\cdots a_1$
which is equal to $\ldd(c'c_1)$ (resp., $\rdd(c'c_1)$). Now,
$\ldd_2(c'c_1)=a_k\cdots
a_{s+1}$ is a column word and a prefix of $c_1$. Note that when $s=k$, then $\ldd(c'w)=c'c_1=\, \ldd(c'c_1)$ (resp., $\rdd(c'w)=c'c_1=\, \rdd(c'c_1)$) and $\ldd_2(c'c_1)$ (resp., $\rdd_2(c'c_1)$) is the empty word.
\end{proof}

The following algorithm gives us a left insertion method of a symbol $a\in \aA$
into a \ps\ tableau $B$, producing a new tableau (more precisely a canonical word)
which we denote by $a\rightarrow B$. The algorithm proceeds from left to right
on the columns of $B$.

\begin{algorithm}[Left insertion of a symbol on a \ps\ tableau]\label{alg:left_insertion}
~\par\nobreak
\textit{Input:} A symbol $a\in\aA$ and an \lps\ (resp., \rps) tableau $B$.

\textit{Output:} The \lps\ (resp., \rps) tableau $a\rightarrow B$.

\textit{Method:} Let $c'_1$ denote the column $a$, and let $(c_1,\ldots, c_k)$ be
the \lps\ (resp., \rps) column configuration of $\cC(B)$.
For each $i=1,\ldots, k$ proceed as follows:
\begin{itemize}
	\item[Step $i$:] Denote by $d_{i}$
the \lps\ (\rps) decreasing subsequence $\ldd(c'_{i}c_i)$ (resp., $\rdd(c'_{i}c_i)$) and let
$c'_{i+1}$  be the prefix of $c_i$ that is obtained from $c'_{i}c_i$ deleting the symbols from $d_{i}$.
\end{itemize} Let $d_{k+1}=c'_{k+1}$ and output the  word $d_1\cdots d_kd_{k+1}$.
\end{algorithm}

One of the following two cases can occur during the computation of the $i$-th step of the previous
algorithm:
\begin{description}
 \item[Case A] $c'_{i}c_i$ is not a canonical word, so the
algorithm computes $d_i$ and  a non-empty column $c'_{i+1}$, and
the minimum symbol of $d_i$ is the minimum
symbol of $c_i$, by the previous lemma;
\item[Case B] $c'_ic_i$ is a canonical word, and so  two cases can occur
according to
the previous lemma:
\begin{enumerate}
 \item[(B1)]   $d_i=c'_ic_i$ is a single column word, and so
$c'_{i+1}$ is the empty word and the algorithm will compute $d_{j}=c_j$, for $
i<j\leq k$, and $d_{k+1}$ is the empty word; or,
\item[(B2)]  $d_i=c'_i$, and so
$c'_{i+1}=c_i$ and the algorithm computes
$d_{j+1}=c_{j}$, for $i\leq j\leq k$.
\end{enumerate}
\end{description}

Notice that once the algorithm reaches Case B  it maintains in that
case, and never returns to Case A in the next iterations. So one of the
following situations occurs in the execution of the algorithm:
\begin{description}
 \item[Situation (1)] If Case A occurs for all $i$, then the algorithm will
output a word
that has column configuration $(d_1,\ldots, d_k,d_{k+1})$;
 \item[Situation (2)] If  the algorithm first enters on Case B1 in the
$i$-th iteration, then
it will output a word with column configuration $(d_1,\ldots,d_{i-1},
d_i(=c'_ic_i), d_{i+1}(=c_{i+1}),\ldots, d_k(=c_k))$;
 \item[Situation (3)] As if the algorithm first
enters Case B2 in the $i$-th iteration, then it will output a word with
column configuration $(d_1,\ldots,d_{i-1},
d_i(=c'_i), d_{i+1}(=c_{i}),\ldots, d_{k+1}(=c_k))$.
\end{description}

\begin{lem}\label{lem:left_insertion}
 Algorithm~\ref{alg:left_insertion} outputs a canonical word. With the notation
used in the algorithm, the column configuration of the output is $(d_1,\ldots,
d_k,d_{k+1})$, with $d_{k+1}$ possibly empty.
\end{lem}
\begin{proof}
Observe that on each iteration, the word $d_ic'_{i+1}$ is a canonical word with
column configuration $(d_i,c'_{i+1})$.

If Situation (2) occurs, then the minimum symbol of $d_i$ is the minimum symbol of
$c_i$, for $1\leq i\leq k$. Similarly for Situation (1) in which
$c'_{k+1}=d_{k+1}$, and so $d_kd_{k+1}$ is a canonical word. Therefore, the
algorithm outputs a canonical word.

If Situation (3) occurs, then $c'_i=d_i$ and thus using the observation on the beginning of the proof $d_{i-1}d_i$ is a canonical  word with column configuration $(d_{i-1},d_i)$. Similarly, $(d_i, d_{i+1})$ is the co\-lumn configuration of the canonical word $d_id_{i+1}$, since $c'_{i+1}=c_i$.
Since also $(d_1,\ldots, d_{i-1})$ (with the minimum symbol of $d_j$ equal to
the minimum symbol of $c_j$, for $1\leq j< i$, as in Case A) and
$(d_{i+1}(=c_{i}),\ldots, d_{k+1}(=c_k))$ are column configurations of the
corresponding words, we conclude that also in this situation the algorithm
outputs a canonical word.
\end{proof}

Below, we provide two examples of the execution of the left insertion algorithm for an \lps\ tableau:
\begin{figure}[ht]
	\centering
	\begin{align*}
	 &\ytableausetup
	{mathmode, boxsize=1.3em, aligntableaux=center}
	\left(\begin{ytableau}
	*(red!80!black)4
	\end{ytableau}\rightarrow\begin{ytableau}
	*(green!70!black)6 & \none & \none & 5 & \none \\
	*(green!70!black)4 & 8 & 7 & 4 & 8 \\
	*(gray!60!white)2 & 3 & 3 & 3 & 5 \\
	*(gray!60!white)1 & 1 & 1 & 2 & 3
	\end{ytableau}\right)=\begin{ytableau}
	4\\
	2\\
	1
	\end{ytableau}\cdot\left(\begin{ytableau}
	*(red!80!black)6\\
	*(red!80!black)4
	\end{ytableau}\rightarrow\begin{ytableau}
	\none & \none & 5 & \none \\
	*(green!70!black)8 & 7 & 4 & 8 \\
	*(gray!60!white)3 & 3 & 3 & 5 \\
	*(gray!60!white)1 & 1 & 2 & 3
	\end{ytableau}\right)\\
	&=\begin{ytableau}
	\none & 6\\
	4 & 4\\
	2 & 3\\
	1 & 1
	\end{ytableau}\cdot\left(\begin{ytableau}
	*(red!80!black)8
	\end{ytableau}\rightarrow\begin{ytableau}
	\none & 5 & \none \\
	*(gray!60!white)7 & 4 & 8 \\
	*(gray!60!white) 3 & 3 & 5 \\
	*(gray!60!white)1 & 2 & 3
	\end{ytableau}\right)=\begin{ytableau}
	\none & 6 & 8 & 5 & \none\\
	4 & 4 & 7 & 4 & 8 \\
	2 & 3 & 3 & 3 & 5 \\
	1 & 1 & 1 & 2 & 3
	\end{ytableau}
	\end{align*}
	\caption{Example of Situation (2).}
\end{figure}

\begin{figure}[ht]
	\centering
	\begin{align*}
	\ytableausetup
	{mathmode, boxsize=1.3em, aligntableaux=center}
	&\left(\begin{ytableau}
	*(red!80!black)2
	\end{ytableau}\rightarrow\begin{ytableau}
	*(green!70!black)6 & \none & \none & 5 & \none \\
	*(green!70!black)4 & 8 & 7 & 4 & 8 \\
	*(green!70!black)2 & 3 & 3 & 3 & 5 \\
	*(gray!60!white)1 & 1 & 1 & 2 & 3
	\end{ytableau}\right)=\begin{ytableau}
	2\\
	1
	\end{ytableau}\cdot\left(\begin{ytableau}
	*(red!80!black)6\\
	*(red!80!black)4\\
	*(red!80!black)2
	\end{ytableau}\rightarrow\begin{ytableau}
	\none & \none & 5 & \none \\
	*(green!70!black)8 & 7 & 4 & 8 \\
	*(green!70!black)3 & 3 & 3 & 5 \\
	*(gray!60!white)1 & 1 & 2 & 3
	\end{ytableau}\right)\\
	&=\begin{ytableau}
	\none & 6\\
	\none & 4\\
	2 & 2\\
	1 & 1
	\end{ytableau}\cdot\left(\begin{ytableau}
	*(red!80!black)8\\
	*(red!80!black)3
	\end{ytableau}\rightarrow\begin{ytableau}
	\none & 5 & \none \\
	*(green!70!black)7 & 4 & 8 \\
	*(green!70!black) 3 & 3 & 5 \\
	*(gray!60!white)1 & 2 & 3
	\end{ytableau}\right)=\begin{ytableau}
	\none & 6 & \none\\
	\none & 4 & 8\\
	2 & 2 & 3\\
	1 & 1 & 1
	\end{ytableau}\cdot\left(\begin{ytableau}
	*(red!80!black)7\\
	*(red!80!black)3
	\end{ytableau}\rightarrow\begin{ytableau}
	*(green!70!black)5 & \none \\
	*(green!70!black)4 & 8 \\
	*(green!70!black)3 & 5 \\
	*(gray!60!white)2 & 3
	\end{ytableau}\right)\\
	&=\begin{ytableau}
	\none & 6 & \none & \none\\
	\none & 4 & 8 & 7\\
	2 & 2 & 3 & 3\\
	1 & 1 & 1 & 2
	\end{ytableau}\cdot\left(\begin{ytableau}
	*(red!80!black)5\\
	*(red!80!black)4\\
	*(red!80!black)3
	\end{ytableau}\rightarrow\begin{ytableau}
	*(green!70!black)8 \\
	*(green!70!black)5 \\
	*(green!70!black)3
	\end{ytableau}\right)=\begin{ytableau}
	\none & 6 & \none & \none & \none & \none\\
	\none & 4 & 8 & 7 & 5 & 8\\
	2 & 2 & 3 & 3 & 4 & 5\\
	1 & 1 & 1 & 2 & 3 & 3
	\end{ytableau}
	\end{align*}
	\caption{Example of Situation (3).}
\end{figure}

\begin{lem}\label{lem:left_bell_insertion}
 For any word $w\in \aA^*$ and $a\in \aA$, we have
 \[\big(a\rightarrow \bB(w)\big) = \bB(aw).\]
\end{lem}
\begin{proof}
Maintaining the notation used in Algorithm~\ref{alg:left_insertion}, with the
input $a$ and $B=\bB(w)$, the procedure outputs a canonical word with column
configuration $(d_1,\ldots,d_k,d_{k+1})$, by Lemma~\ref{lem:left_insertion}.

Next we show that the $i$-th decreasing subsequence of $aw$ is the word $d_i$,
thus proving the result. The initial decreasing subsequence $\dd_1(aw)$ of $aw$ is
computed by taking $\dd(ac_1\cdots c_k)$. By
Lemma~\ref{lem:initial_decreasing_sequence}, $\dd_1(aw)=\dd(ac_1)$, thus equal
to $d_1$, since $d_1=\dd(c'_1c_1)$ and $a=c'_1$. With the notation used in
Algorithm~\ref{alg:left_insertion}, $c'_2$ is the subsequence of $ac_1$ from which
the symbols of  $\dd_1(aw)$ were erased. Thus to compute $\dd_2(aw)$ we compute
$\dd(c'_2c_2\cdots c_k)$. By
Lemma~\ref{lem:initial_decreasing_sequence}, $c'_2$ is a column word. Hence,
the
same reasoning can now be applied to compute $\dd_i(aw)$, knowing that
$\dd_j(aw)=d_j$, for $1\leq j<i$, and $\dd_i(aw)=\dd(c'_ic_i\cdots
c_k)=\dd(c'_ic_i)$, by Lemma~\ref{lem:initial_decreasing_sequence}, for $
i\leq k$.  Finally, if $c'_{k+1}$ is non empty, these are the non-erased
letters after the computation of ${\dd}_k(aw)$, and since it is a column word
we get $\dd_{k+1}(aw)=c'_{k+1}=d_{k+1}$.
\end{proof}

For $x \in \{\ell,r\}$, given $w=w_1\cdots w_n\in\aA^*$, with $w_1\ldots,w_n\in \aA$, the corresponding $x$\ps\ tableau
can be computed using right-algorithm by:
\[
	\bB_x(w)=\Big(\big((\emptyset \leftarrow w_1) \leftarrow w_2 \big)\ldots
	\Big)\leftarrow w_n.
\]
The previous lemma provides us an alternative way:
\[
  \bB_x(w)=w_1\rightarrow\Big(\ldots\big(w_{n-1}\rightarrow(w_n\rightarrow\emptyset)\big)\Big).
\]
More formally, the algorithm is defined as follows:

\begin{algorithm}[Left \ps\ algorithm for words]\label{alg:left_insertionword}
	~\par\nobreak
	\textit{Input:} A word $w=w_1\cdots w_n\in \aA^*$, with $w_1,\ldots,w_n\in \aA$.

	\textit{Output:} An \lps\ (resp., \rps) tableau $\bB_\ell(w)$ (resp., $\bB_r(w)$).

	\textit{Method:}
	Start with the empty tableau $P_0=\emptyset$. For each $i=1,\ldots,k$, insert $w_{n-k+1}$ into the \lps\ (resp., \rps) tableau $P_{i-1}$ as per Algorithm~\ref{alg:left_insertion}. Output $P_k$ for $\bB_\ell(w)$ (resp., $\bB_r(w)$).
\end{algorithm}

\subsection{Complexity of the algorithms}
Regarding the plactic monoid, it is known that  Schensted's algorithm has time complexity $O\left(n\log (n)\right)$ on the  length $n$ of the input word \cite{FREDMAN197529}.

In this subsection we study the complexity of the algorithms used in this paper to compute the column configuration of a
word, namely the Patience Sorting algorithm (Algorithm~\ref{alg:Bellword}), the left-to-right minimal subsequence
algorithm (Algorithm~\ref{alg:lefttoright}), and the Left \ps\ algorithm (Algorithm~\ref{alg:left_insertionword}).  (See
\cite{Sipser:1996:ITC:524279} for background on complexity.)

\begin{prop}
	The Patience Sorting algorithm, Algorithm~\ref{alg:Bellword}, has time complexity $O\left(n\log (n)\right)$,
	where $n$ is the length of the input word.
\end{prop}

\begin{proof}
  Let $w\in \aA^*$, with $|w|=n$. Whenever we are inserting a symbol of $w$ using Algorithm~\ref{alg:Bell}, we can apply
  the binary search algorithm to the bottom row (which is sorted increasingly) to find the position in which this symbol
  is going to be inserted. There can be at most $n$ different
  symbols on the bottom row, and the time complexity of the binary search is thus $O\left(\log (n)\right)$ (cf. \cite[Ch. 2, Ex. 2.3-5]{MR2572804}
  and \cite{Flores:1971:ABS:362663.362752}). We apply this search $|w| = n$ times, so
  the time complexity for computing $\bB_\ell(w)$ (resp., $\bB_r(w)$) using Algorithm~\ref{alg:Bellword} is
  $O\left(n\log (n)\right)$.
\end{proof}

The idea behind the construction of Algorithm~\ref{alg:lefttoright} is the use of the decreasing subsequences algorithm (Algorithm~\ref{alg:decreasing}) in order to iteratively construct the \ps\ column words of the column configuration of $w$.


\begin{prop}
	The time complexity of Algorithm~\ref{alg:lefttoright} is
$O\left(n^2\right)$, where $n$ is the length of the input word.
\end{prop}
\begin{proof}
	The decreasing subsequence construction is based on a left to right search on the
word comparing symbols to find the next symbol less than (resp., less than or equal on the $r$-case) than the previous one. Given a word $w\in \aA^*$ of length $n$, on the first step the algorithm chooses the leftmost symbol of $w$, as for the second step it can do, in the worst-case scenario, $n-1$ comparisons. In the worst-case scenario, on the $i$-th step the algorithm can perform $n-i+1$ comparisons.
Therefore, the time complexity of the algorithm is
	\begin{equation*}
	O\left((n-1)+(n-2)+\cdots +1\right)=O\left(\frac{n(n-1)}{2}\right)=O\left(n^2\right).\qedhere
	\end{equation*}
\end{proof}
\begin{prop}
	Left \ps\ algorithm, Algorithm~\ref{alg:left_insertionword}, has time complexity $O\left(n^2\right)$, where $n$ is the length of the input word.
\end{prop}
\begin{proof}
Given a word of length $n$,
Algorithm~\ref{alg:left_insertionword} executes Algorithm~\ref{alg:left_insertion} $n$ times.

	Let $w=w_n\cdots w_1\in \aA^*$. We shall now analyse the time complexity of Algorithm~\ref{alg:left_insertion} when inserting a symbol $w_{s+1}$ into the tableau $\bB(w_{s}\cdots w_1)$. Let $(c_1,\ldots, c_{k_s})$ be the column configuration of $\cC(\bB(w_{s}\cdots w_1))$ and let $\Sh(\bB(w_s\cdots w_1))=(j_1,\ldots, j_{k_s})$, where $j_1+\cdots +j_{k_s}=s$.

	To execute Algorithm~\ref{alg:left_insertion} we need only to compute the decreasing subsequences $\dd(c'_ic_i)$, $k_s$ times. (We keep the notation used in the algorithm.) Note that $c'_i$ and $c_i$ are column words. As such we can compute $\dd(c_i'c_i)$ and $c'_{i+1}$ scanning through the word $c_i'c_i$ from left to right only once. Recall that $c'_{i+1}$ is a prefix of $c_i$. So reading the word $c'_ic_i$ we will find first the symbols from $c'_i$, and if we find either an increase or an equal symbol (an increase on the $r$ case), it means that we reached the first symbol of $c'_{i+1}$; keeping in  memory the rightmost symbol, say $a$, of $c'_i$ we can proceed through the word reading now symbols from $c'_{i+1}$; the word $c'_{i+1}$  finishes right before  we find the first symbol less (less or equal on the $r$ case) than $a$. Having identified $c'_{i+1}$ we also get  $\dd(c_i'c_i)$.

	Since $|c'_1|=1$ and $|c'_{i+1}|\leq |c_i|= j_i$, we deduce that Algorithm~\ref{alg:left_insertion}, needs to do at most $|c'_1|+ |c'_2\cdots c'_{{k_s}+1}| + |c_1\cdots c_{k_s}|\leq 2s+1$ comparisons to execute the insertion of symbol $w_{s+1}$ into the tableau $\bB(w_{s}\cdots w_1)$.  Therefore, it has time complexity $O(s)$.
Hence, Algorithm~\ref{alg:left_insertionword} has time complexity
	$O\left(1+\cdots+n\right)=O\left(n^2\right)$.
\end{proof}

\subsection{\ps\ congruences}\label{subsec:Bell_congruence}

In Subsection~\ref{section1.1}, we defined \lps\ and \rps\ tableaux. In each case we can identify words that lead to the
same \ps\ tableau, which yields, as will be shown, congruences on $\aA^*$.  We then give sets of relations that
generate each congruence.

Define relations $\bellcong$ and $\belrcong$ by
\begin{align*}
u\bellcong v &\iff \bB_\ell(u)=\bB_\ell(v), \\
u\belrcong v &\iff \bB_r(u)=\bB_r(v),
\end{align*}
for $u,v\in \aA^*$. By $\bB_\ell(u)=\bB_\ell(v)$.

By \cite[Theorem~3.1]{Maxime07}, the relation $\bellcong$ is a congruence. We present a short proof that both
$\bellcong$ and $\belrcong$ are congruences using Algorithm~\ref{alg:Bellword} and Lemma~\ref{lem:left_bell_insertion}.

\begin{prop}\label{prop115}
  The relations $\bellcong$ and $\belrcong$ are congruences on $\aA^*$.
\end{prop}

\begin{proof}
Let $x \in \{\ell,r\}$ and let $\belxcong$ denote the corresponding relation $\bellcong$ or $\belrcong$.
Let $u,v\in \aA^*$ be such that $u\belxcong v$. We will prove that
$ua\belxcong va$ and that $au\belxcong av$, for any $a\in \aA$. Then the result
follows
immediately by an inductive argument on the length of a word
$w=\in\aA^*$.

Let $a\in \aA$. Since $\bB_x(u)=\bB_x(v)$, by Algorithm~\ref{alg:Bellword} we get
\[\bB_x(ua) = \big(\bB_x(u)\leftarrow a\big)= \big(\bB_x(v)\leftarrow a\big)= \bB_x(va)\]
and hence $ua\belxcong va$.

Now, by Lemma~\ref{lem:left_bell_insertion} we obtain
\[\bB_x(au)= \big(a\rightarrow\bB_x(u)\big) = \big(a\rightarrow\bB_x(v)\big) =\bB_x(av)\]
and hence $au\belxcong av$.
\end{proof}

Thus, we define the \lps\ monoid and the \rps\ monoid, denoted by $\bell$ and $\belr$, respectively, as the quotients of the free monoid $\aA^*$ over the congruences $\bellcong$ and $\belrcong$, respectively.
The \lps\ monoid is also known as the Bell monoid \cite{Maxime07}.

So, each \ps\ (\lps\ or \rps) tableau, and therefore each \ps\ canonical word, is a unique representative of a \ps\ class.

In the following paragraphs we present the analogue of Knuth's relations for these monoids.
Consider the following binary relations on $\aA^*$:
\begin{align*}
 \rR_\bell&=\{\,(yux,yxu): m\in \mathbb{N},\, x,y,u_1,\ldots , u_m\in
\aA, \\
 &\qquad   u=u_m\cdots u_1,\, x<y\leq
u_1< \cdots < u_m\,\}
\end{align*}
and
\begin{align*}
\rR_\belr&=\{\,(yux,yxu): m\in \mathbb{N},\, x,y,u_1,\ldots , u_m\in
\aA, \\
&\qquad  u=u_m\cdots u_1,\, x\leq y<
u_1\leq \cdots \leq u_m\,\}.
\end{align*}

In the remainder of this subsection we prove that $\bellcong$ (resp., $\belrcong$) is equal to
$\rR_\bell^{\#}$ (resp., $\rR_\belr^{\#}$), the smallest congruence relation containing $\rR_\bell$ (resp., $\rR_\belr$).
Let us begin with an auxiliary result.

\begin{lem}
  \label{prop19}
  For any $w\in \aA^*$, if $w_c$ is the \lps\ (resp., \rps) canonical word associated to $w$, then $(w, w_c) \in \rR_\bell^{\#}$ (resp., $(w, w_c) \in \rR_\belr^{\#}$).
\end{lem}

\begin{proof}
We give the detailed proof for the $\ell$-case and in parentheses give the main differences for the $r$-case.

We prove the result by induction on the length of $w$.
If $w$ is a symbol in $\aA$, the  result follows immediately since $w=w_c$.

Suppose that the result holds for any word of length $n$. Let $w\in \aA^*$ be a word of length $n$ and $a\in \aA$. Suppose
also that $\ldD(w)=(c_1,\ldots, c_k)$, and let $a_i$ be the minimum symbol in $c_i$, for $1\leq i\leq k$. Note that $a_i$ is the
rightmost symbol in the word $c_i$, and that $a_1\leq \cdots \leq a_k$ (resp., $a_1< \cdots < a_k$).

If $a_k\leq a$, then $\cC\big(\bB_\ell(wa)\big)=c_1\cdots c_ka=w_ca$ and the result follows trivially.
Otherwise, let $j\in \{1,\ldots,k\}$ be the smallest index such that $a_1\leq \cdots \leq
a_{j-1}\leq a < a_j \leq \cdots \leq a_k$ (resp., $a_1< \cdots < a_{j-1}< a \leq a_j < \cdots < a_k$). Then $\ldD(wa)=(c_1,\ldots , c_{j-1}, c_ja, c_{j+1},\ldots ,c_k)$.

Now, we have $(a_ic_{i+1}a
, a_iac_{i+1})\in \rR_\bell$, for $j\leq i\leq k$, since $a<a_i\leq a_{i+1}$ (resp., $a\leq a_i < a_{i+1}$) and $c_{i+1}$ is an \lps\ column word. By the induction hypothesis $w\ \rR_\bell^{\#}\ \cC\big(\bB_\ell(w)\big)$. As
\begin{align*}
wa\ &\rR_\bell^{\#}\  \cC\big(\bB_\ell(w)\big)a=c_1\cdots c_jc_{j+1}\dots c_{k-1}c_k a\\
     &\rR_\bell^{\#}\ c_1\cdots c_jc_{j+1}\cdots c_{k-1}ac_k\\
     & \cdots \\
     & \rR_\bell^{\#}\ c_1\dots c_jac_{j+1}\dots c_k=\cC\big(\bB_\ell(wa)\big),
\end{align*}
the result follows.
\end{proof}

\begin{thm}
	\label{thm119}
The relations $\bellcong$ and $\belrcong$ on $\aA^*$ are, respectively, the smallest congruences generated by the relations $\rR_\bell$ and $\rR_\belr$.
\end{thm}
\begin{proof}
  Suppose $(w, w')\in \rR_\bell$. Then $w=yux$ and $w'=yxu$, for some \lps\ column word $u=u_m\cdots
  u_1$. 
  Because $\ldD(w)=(yx, u)=\ldD(w')$, it follows that $\bB_\ell(w)=\bB_\ell(w')$ and so $w\bellcong w'$. Therefore,
  $\rR_\bell^{\#}$ is contained in the congruence $\bellcong$.

Conversely, let $u,v \in \aA^*$ be such that $u\bellcong v$. By definition
$\bB_\ell(u)=\bB_\ell(v)$ and so $\cC\big(\bB_\ell(u)\big)=\cC\big(\bB_\ell(v)\big)$. Using Lemma~\ref{prop19}, we
conclude that $(u,v) \in \rR_\bell^{\#}$ by symmetry and transitivity of $\rR_\bell^{\#}$.

The \rps\ part follows by the same argument using the appropriate \rps\ definitions.
\end{proof}

The \ps\ monoids, $\bell$ and $\belr$, can now be described as the quotients of the free
monoid over the alphabet $\aA$ by the congruence generated by the relations
$\rR_\bell$ and $\rR_\belr$, respectively. So, we can conclude that $\bell$ and $\belr$ are defined by the presentations $(\aA,\rR_\bell)$ and $(\aA,\rR_\belr)$, res\-pec\-tively.

Since for any $a\in \aA$, $(u,v)\in \rR_\bell$ and $(u',v')\in \rR_\belr$, we have that $|u|_a=|v|_a$ and
$|u'|_a=|v'|_a$, the presentations $(\aA,\rR_\bell)$ and $(\aA,\rR_\belr)$ are multihomogeneous presentations and
therefore $\bell$ and $\belr$ are multihomogeneous monoids.

Throughout the text we identify words over $\aA$ with
elements of the monoids $\bell$ and $\belr$ that they represent.

\subsection{\ps\ monoids of finite rank}

Recall that  $\aA_n=\{1<\cdots <n\}$. The restriction of the relations
$\bellcong$ (resp., $\belrcong$) and $\rR_{\bell}$ (resp., $\rR_{\belr}$) to the alphabet $\aA_n$ yields the \lps\ (resp., \rps) \emph{congruence of rank $n$} denoted by ${\bellcongn}$ (resp., ${\belrcongn}$) and the set of $\bell$ (resp., $\belr$) relations $\rR_{\bell_n}$ (resp., $\rR_{\belr_n}$).

In a similar way to Theorem~\ref{thm119} we
conclude that ${\bellcongn}$ (resp., ${\belrcongn}$) is the smallest congruence relation on
$\aA_n^*$ generated by $\rR_{\bell_n}$ (resp., $\rR_{\belr_n}$).
We define the \lps\ (resp., \rps) \emph{monoid of rank} $n$, denoted by $\bell_n$ (resp., $\belr_n$), to be the quotient of the free monoid on the alphabet $\aA_n$ by the congruence ${\bellcongn}$ (resp., ${\belrcongn}$). The pair $(\aA_n ,\rR_{\bell_n})$ is  a presentation of the monoid
$\bell_n$, whereas $(\aA_n ,\rR_{\belr_n})$ is a presentation for $\belr_n$.

Since $\aA_n$ is finite, there are only finitely many words $u=u_m\cdots u_1\allowdisplaybreaks\in \aA_n^*$, with $u_1,\ldots,u_m\in \aA_n$ such that $u_1<\ldots <u_m$. As there are also finitely many possibilities for symbols $x,y\in \aA_n$, the set $\rR_{\bell_n}$ is finite. So the presentation $(\aA_n
,\rR_{\bell_n})$ is finite and we can conclude the following:
\begin{prop}
	For any $n\in \mathbb{N}$, the \lps\ monoid of rank $n$, $\bell_n$, is
finitely presented.
\end{prop}

Since there are infinitely many decreasing words $u=u_m\cdots u_1\in \aA_n^*$, with $u_1,\ldots,\allowbreak u_m\in \aA_n$ such that $u_1\leq\cdots \leq u_m$, the set $\rR_{\belr_n}$ is infinite and thus the presentation $(\aA_n ,\rR_{\belr_n})$ is infinite.


\begin{prop}
  For any $n\in \mathbb{N}$ with $n\geq 2$, the monoid $\belr_n$ is not finitely presented.
\end{prop}

\begin{proof}
Let $n\in \mathbb{N}$ be such that $n\geq 2$. Suppose, with the aim of obtaining a contradiction, that $\belr_n$
is finitely presented. Then, there exists a finite presentation for $\belr_n$
using the generating set $\aA_n$ \cite[Proposition 3.1]{ruskucphdthesis}. Let
$(\aA_n, \mathcal{G})$ be such a presentation. Since $n\geq 2$, the relation $12^i1\belrcongn 112^i$ holds, for any $i\in \mathbb{N}$.

The word $12^i1$ can have factors of three possible types of words: $12^k$, $2^k$ or $2^k1$ (where $k\leq i$ can vary). With the generating set $\aA^*_n$, none of the given types of words satisfies a non-trivial relation, meaning that if for example $12^k= w$ in $\belr_n$, then $12^k$ and $w$ are equal as words.

 We conclude that for each $i\in \mathbb{N}$, $\mathcal{G}$
must have a relation whose right-hand side or left-hand side is equal to $12^i1$.
Thus $\mathcal{G}$ is infinite, which a contradiction.
\end{proof}


\section{Growth and identities}
\label{sec:growth}

\subsection{Growth of \ps\ monoids of finite rank}

In this subsection we are interested in studying the growth of the
finitely ranked \lps\ and \rps\ monoids.
As we will see the \lps\ and \rps\ monoids of finite rank have,
in general, distinct type of growths.

Given any monoid $M$ generated by a finite set $\Sigma$, the \emph{growth function} of $M$ with respect to $\Sigma$, is
the function $\gamma_M:\mathbb{N}\to \mathbb{N}$ where $\gamma_M(N)$ is the number of elements of $M$ that can be
expressed as products of length at most $N$ of generators from $\Sigma$.

For functions $f,g:\mathbb{N}\to \mathbb{N}$, we write $f\preceq g$ if there exists a positive constant $c$ such that
for all sufficiently large $N\in \mathbb{N}$, we have $f(N)\leq g(cN)$. The functions $f$ and $g$ are said to be
equivalent if $f\preceq g$ and $g\preceq f$. We shall identify a function $f$ with its equivalence class.

For any given finitely generated monoid, the growth functions relative to two distinct finite generating sets are
equivalent (see \cite[Section~9]{Otto1997} and \cite{sapir2014combinatorial} for more details on the growth of
functions). As such, we consider equivalence classes of functions to study the growth of monoids.

With the above definition, we say that a finitely generated monoid $M$ has \emph{polynomial growth} if $\gamma_M$ is
bounded above by a function $N^k$, for some $k\in \mathbb{N}$.  We say that the growth of $M$ is \emph{exponential}, if
$\gamma_M$ is bounded by below by a function $k^N$ for some $k \in \mathbb{N}$.

Note that the free monoid of rank $1$ has
polynomial growth, and for any $n\in \mathbb{N}$, with $n\geq 2$, the free
monoid of rank $n$ has exponential growth.

It is clear that both $\bell_1$ and $\belr_1$ are free monoids of rank $1$, and
hence have polynomial growth.

\begin{prop}\label{prop:embedding_into_bell}
For any $n\in \mathbb{N}$, with $n\geq 2$, the free monoid of rank $2^{n-1}$ embeds into $\bell_n$.
\end{prop}

\begin{proof}
Let $C_n$ 
be the subset of $\aA_n^*$ of \lps\ column words in which any
of its elements has rightmost symbol $1$.
In $\aA_n$ there are $n-1$ elements greater than $1$. So, as \lps\ column words are strictly decreasing
words, we can conclude that there are
$2^{n-1}$ elements in
$C_n$.

Let $c_1\cdots c_k$ be a product of elements of $C_n$. By induction on
$i$, it is straightforward to see that the tableau $\bB_\ell(c_1\cdots c_i)$ only has symbols
$1$ (the last symbol of each element of $C_n$) on its bottom row, and thus that
$c_{i+1}$ is inserted as a new column at the right of the tableau. Hence the $i$-th column of $\bB_\ell(c_1\cdots
c_k)$ is simply $c_i$. Since each term of a product of elements of
$C_n$ can be recovered from the tableau, it follows that the elements of $C_n$ generate a free submonoid of $\bell_n$.
\end{proof}

Using the fact that if $M_1$ is a submonoid $M_2$, then $\gamma_{M_1}\preceq \gamma_{M_2}$
\cite{sapir2014combinatorial}, we deduce the following:

\begin{cor}\label{exponentialgrowth}
	The \lps\ monoids of rank greater than $1$ have exponential growth.
\end{cor}

Regarding the monoids $\belr_n$, for $n\geq 2$, we will show that
they have polynomial growth.

\begin{thm}
	For any $n\in \mathbb{N}$, the monoid $\belr_n$ has polynomial growth.
\end{thm}
\begin{proof}
  Let $n,N\in \mathbb{N}$ and consider the growth function $\gamma_{\belr_n}$. The value $\gamma_{\belr_n}(N)$ indicates the
  number of \rps\ tableaux with at most $N$ entries. Notice that each such \rps\ tableau has at most $n$ columns and
  each column has at most height $N$.

  A column of height at most $N$ is filled with symbols from the set $\aA_n$.  The number of weakly decreasing words
  over $\aA_n$ of length at most $N\in\mathbb{N}$ is uniquely determined by the number of solutions to
  $m_1+\cdots+m_n\leq N$, where for any $i\in \aA_n$, $m_i$ denotes the (non-negative) number of $i$'s occurring in the
  considered word of length at most $N$.  The number of solutions to this inequality is given by $\binom{n+N}{N}$
  \cite[Subsection~1.2]{Stanley2011}.  Thus there are at most $\binom{n+N}{N}$ different columns of height at most $N$.
  Thus there are at most
  \[
    n\,\binom{n+N}{N}
  \]
  \rps\ tableaux with at most $N$ entries.

  Notice that $\binom{n+N}{N}\leq \binom{n+N}{N} n!= (n+N)(n+N-1)\cdots (N +1)\leq (n+ N)^{n}$, and therefore
  $\gamma_{\belr_n}$ is bounded above by a polinomial function with non-negative coefficients (This is equivalent to
  have polynomial growth, cf. \cite{deLuca:2011:FRS:2408037}).
\end{proof}

\subsection{Identities on \ps\ monoids}

An \emph{identity} is a formal equality between words of the free monoid over a countable alphabet $\Sigma$. For any monoid
$M$, we say that $M$ satisfies the identity $u=v$ if for any morphism $f:\Sigma^*\to M$, $f(u)=f(v)$, that is, for any
substitution of symbols in $u$ and $v$ by elements of $M$, the equality holds in $M$. For instance, the identity $xy=yx$
is satisfied by any commutative monoid and the bicyclic monoid satisfies Adjan's identity $xyyxxyxyyx = xyyxyxxyyx$
\cite{MR0218434}. The \lps\ and \rps\ monoids are multihomogeneous, therefore the left side and the right side of any
identity satisfied by them has the same length (in fact the same number of each symbol). Thus, we will say that an
identity has length $n\in\mathbb{N}$ if the word on the left side of the identity has length $n$.

Note first that both $\bell_1$ and $\belr_1$ are isomorphic to the free
monoid of rank $1$, which is commutative and therefore satisfies the
non-trivial identity $xy=yx$.

As free monoids of rank greater or equal than $2$ do not satisfy
non-trivial identities, the following proposition is a consequence of
Proposition~\ref{prop:embedding_into_bell}.

\begin{cor}
  The monoids $\bell$ and $\bell_n$ for $n \geq 2$ do not satisfy non-trivial identities.
\end{cor}


This result is contrary to what happens in other multihomogeneous monoids like the sylvester, the hypoplactic, the
Chinese, the Baxter, the stalactic, or the taiga monoid, all of which satisfy non-trivial identities
\cite{1611.04151}. As we will see in the following paragraphs, there is a hierarchy of different identities satisfied by
finite-rank \rps\ monoids, but none satisfies by the infinite-rank \rps\ monoid. (Compare the situation for plactic
monoids: it is known that the infinite-rank plactic monoid does not satisfy a non-trivial identity, but it is not known
whether plactic monoids of rank greater than $3$ satisfy non-trivial identities \cite{ckkmo_placticidentity}.)

Regarding $\belr_n$, with $n\geq 2$, we first  prove some auxiliary lemmata.

\begin{lem}\label{lem:identity1}
	Let $n\in \mathbb{N}$ and $w\in \aA^*_{n}$. Suppose that $w$ has exactly $k$ different symbols $i_1<\dots<i_{k}$ (so $|w|\geq k$). Then, for $j\leq k$, the \rps\ bottom row word of the tableau $\bB_r(w^j)$ has prefix $i_1i_2\cdots i_j$.
\end{lem}

\begin{proof}
We proceed by induction on $j$. The result holds trivially for $j=1$ since $i_1$ is the smallest symbol in $w$ and so the leftmost column of $\bB_r(w)$ has $i_1$ as the bottom symbol.

Now suppose that the result holds for $j<k$, and so  the  bottom row word of the tableau $\bB_r(w^j)$ has prefix $i_1i_2\cdots i_j$.
Denote by $a_1$ the symbol in $w^j$ that takes the position of $i_j$ in the bottom row of $\bB_r(w^j)$. Since  $w$ has $k$ different symbols then denote by $a_2$ the rightmost symbol $i_{j+1}$ in $w^{j+1}$.
As result, $a_1a_2$ is a subsequence of $w^{j+1}$. By Lemma~\ref{lemma:equal_symbols_in_tableau} and since $a_1<a_2$, the symbol $a_2$ is positioned in $\bB_r(w^{j+1})$ in a column further to the right of the column containing $a_1$.

Resuming, there is a column, say $c$,  further to the right of the column containing $a_1=i_j$ which has the symbol $i_{j+1}$. Hence, $c$ has a bottom symbol $a$ which satisfies $i_j<a\leq i_{j+1}$. Since there is no other symbol from $w$ between $i_j$ and $i_{j+1}$, we deduce that $a=i_{j+1}$ and that $\bB_r(w^{j+1})$ has bottom row word with prefix $i_1i_2\cdots i_ji_{j+1}$.
\end{proof}


\begin{lem}\label{lem:identity2}
Let $n\in \mathbb{N}$ and $w\in \aA^*_{n}$. Suppose that $w$ has exactly $k$
different symbols $i_1<\dots<i_{k}$ and that $\bB_r(w)$  has bottom row word
$i_1\cdots i_k$. Then, for any words $p,q\in\{i_1,\ldots,i_k\}^*$, with
$\ev{p}=\ev{q}$, we get
\[\bB_r(wp)=\bB_r(wq).\]
\end{lem}

\begin{proof}
The result follows from the execution of Algorithm~\ref{alg:Bellword} to compute
$\bB_r(w)\leftarrow p$ and $\bB_r(w)\leftarrow q$. Each symbol $i_j$ (in $p$ or
$q$) will be inserted in the column with bottom symbol $i_j$. Therefore, since
$\ev{p}=\ev{q}$ we get the intended result.
\end{proof}

\begin{prop}\label{identityrps}
	For any $n\in \mathbb{N}$, the monoid $\belr_{n}$ satisfies the identity:
	\begin{equation*}
		(xy)^{n+1}=(xy)^{n}yx.
	\end{equation*}
\end{prop}
\begin{proof}
Let $u,v\in\aA^*_n$. Denote by $i_1< \dots <i_k$ all the symbols from $\aA$ in
the word $uv$. Note that $\ev{uv}=\ev{vu}$.

By Lemma~\ref{lem:identity1}, the tableau
$\bB_r((uv)^k)$  has bottom row word  $i_1\cdots i_k$. Hence, by
Lemma~\ref{lem:identity2}, we obtain $\bB_r((uv)^{n+1})=\bB_r((uv)^{n}vu)$,
since $(uv)^{n+1}=(uv)^k(uv)^{n+1-k}$ and $(uv)^{n}vu=(uv)^k(uv)^{n-k}vu$, and
$\ev{(uv)^{n+1-k}}=\ev{(uv)^{n-k}vu}$.
\end{proof}

\begin{prop}
  \label{prop:rpsnidminlength}
  The monoid $\belr_n$ does not satisfy any non-trivial identity of length less than or equal to $n$.
\end{prop}

The proof of this result follows closely the proof strategy of \cite[Proposition~3.1]{ckkmo_placticidentity} and thus we
only sketch it.

\begin{proof}[Sketch proof]
  Suppose, with the aim of obtaining a contradiction, that $\belr_n$ satisfies a non-trivial identity of length less
  than or equal to $n$. Without loss of generality, assume that it satisfies such an identity over the variable set
  $\{x,y\}$, say $u(x,y) = v(x,y)$, of length equal to $n$ (that is, with $|u| = |v| = n$), where the $j$-th variable
  of $u$ is $x$ and the $j$-th variable of $v$ is $y$.

  Let $s = n(n-1)\cdots 21 \in \aA_n^*$ and let $t = n(n-1)\cdots(j+1)(j-1)\cdots 21 \in \aA_n^*$. (So the word $t$ does not
  contain the generator $j$.)  Since $\belr_n$ satisfies the identity $u(x,y) = v(x,y)$, the tableaux $\bB_r(u(s,t))$
  and $\bB_r(v(s,t))$ are equal, and so the lengths of their bottom rows are equal, which means the lengths of the
  longest strictly increasing subsequences of $u(s,t)$ and $v(s,t)$ are equal. But $u(s,t)$ contains the strictly
  increasing subsequence $12\cdots n$ (since, in particular, $s$ is substituted for the $j$-th variable of $u$), while
  $v(s,t)$ does not contain a strictly increasing subsequence of length at least $n$, since such a subsequence can
  include at most one symbol from each $s$ or $t$ and $t$ was substituted for the $j$-th variable of $v$. This is a
  contradiction.
\end{proof}

Every finite-rank \rps\ monoid $\belr_n$ is a submonoid of the infinite-rank \rps\ monoid $\belr$, so
Proposition~\ref{prop:rpsnidminlength} implies the following result:

\begin{thm}
  The monoid $\belr$ does not satisfy any non-trivial identity.
\end{thm}

\section{Complete presentations}
\label{sec:complete}

Monoids with finite complete presentations are of great importance, and one of
the main reasons is that they provide a solution to the word problem (see
\cite[Subsection~2.2]{MR1215932}).
The goal of this section is to prove that the presentations $(\aA_n, \rR_{\bell_n})$ and $(\aA_n, \rR_{\belr_n})$ are complete. In \cite{Rey2016}, the author presents two complete rewriting systems, but only on the set of permutations.

In order to prove that the presentations are noetherian we will first introduce an
order on $\aA_n^*$.
The length-plus-lexicographic order (also know as shortlex order or radix order)
is an ordering on $\aA_n^*$ where, given
$\alpha=\alpha_k\cdots\alpha_1,\beta=\beta_l\cdots\beta_1 \in \aA_n^*$, with
$\alpha_1,\ldots,\alpha_k,\beta_1,\ldots,\beta_l\in \aA_n$, then $\alpha\ll \beta$
if
\begin{align*}
k<l \vee \left(k=l \wedge \exists i \left(\alpha_i< \beta_i\wedge \forall
j<i\left(\alpha_j=\beta_j\right)\right)\right).
\end{align*}
It is known  that $\ll$ is a well order on
$\aA_n^*$ \cite[Subsection~2.2]{MR1215932}.
\begin{lem}
	\label{lemma1.5}
	For any $n\in \mathbb{N}$, if $\left(w, w'\right)\in \rR_{\bell_n}$ or $\left(w, w'\right)\in \rR_{\belr_n}$, then $w'\ll w$.
	Thus both presentations $(\aA_n,\rR_{\bell_n})$ and $(\aA_n,\rR_{\belr_n})$ are noetherian.
\end{lem}
\begin{proof}
  To prove that $(\aA_n, \rR_{\bell_n})$ (resp., $(\aA_n, \rR_{\belr_n})$) is noetherian we only have to check that
  $w'\ll w$, whenever $\left(w, w'\right)\in \rR_{\bell_n}$ (resp., $\left(w, w'\right)\in \rR_{\belr_n}$), since $\ll$
  is admissible, in the sense of being invariant under left and right multiplication, by
  \cite[Theorem~2.2.4]{MR1215932}. Let $\left(w, w'\right)\in \rR_{\bell_n}$ (resp.,
  $\left(w, w'\right)\in \rR_{\belr_n}$). Then, $w=yux$, $w'=yxu$ for some $u=u_k\cdots u_1\in \aA_n^*$,
  $x,y,u_1,\ldots,u_k \in \aA_n$, $x<y\leq u_1<\ldots<u_k$ (resp., $x\leq y< u_1\leq\ldots\leq u_k$). The words $w$ and
  $w'$ have the same length and the same first symbol, but $x<u_k$, so $w'\ll w$, as required.
\end{proof}

From the previous lemma  we deduce that each $\bell_n$ (resp., $\belr_n$) class has at least one  irreducible element with respect to the relation $\rR_{\bell_n}$ (resp., $\rR_{\belr_n}$) (see \cite[Lemma~2.2.7]{MR1215932}).

\begin{lem}
  For any $n\in \mathbb{N}$, a word is irreducible with respect to the relation $\rR_{\bell_n}$ if and only if it is an
  \lps\ canonical word, and a word is irreducible with respect to the relation $\rR_{\belr_n}$ if and only if it is an
  \rps\ canonical word.
\end{lem}

\begin{proof}
We first of all prove that a canonical word is irreducible. Consider a word $w\in\aA^*_n$ that
it is an \lps\ (resp., \rps) canonical word.
 Let $(c_1,\ldots,c_k)$ be the  \lps\
(resp., \rps) column configuration of $w$. For each $i\in \{1,\ldots, k\}$, let
$a_i$ be the rightmost symbol of $c_i$. Recall that $a_1\leq \cdots \leq a_k$
(resp., $a_1< \cdots < a_k$).

The left-hand side  of a rewriting rule in $\rR_{\bell_n}$ (resp.
$\rR_{\belr_n}$) has the form $yu_m\cdots u_1x$, with
$x,y,u_1,\ldots,u_m\in\aA_n$, and  $x<y\leq u_1< \cdots < u_m$ (resp.
$x\leq y< u_1\leq \cdots \leq u_m$).  So, if a rewriting rule is applied to
$w$, then necessarily $y$ is one of the symbols $a_i$ (the last symbol of a
column), because $y\leq u_m$ (resp., $y<u_m$). Note that if $y$ is the
last symbol of the column $c_j$, then the symbols to its right in the word
$w$ (that is, the symbols in $c_{j+1}\cdots c_k$) are all greater than or equal to (resp.,
strictly greater than) $y$. Hence, no symbol $x$ with $x\leq y$ (resp., $x< y$), is
to right of the symbol $y$ in $w$, as it is the case in the word $yu_m\cdots
u_1x$. Therefore, no rewriting rule can be applied to $w$.

Now we prove that a non-canonical word is reducible.  Let $n\in \mathbb{N}$. Consider a word $w\in\aA^*_n$ that is not
an \lps\ (resp. \rps) canonical word. Let $c_1,\ldots,c_k$ be \lps\ (resp. \rps) column words in $\aA_n^*$ such that
$w=c_1\cdots c_k$, and for any $i\in\{1,\ldots,k-1\}$, $c_ic_{i+1}$ is not an \lps\ (resp. \rps) column word, and for
some $j\in\{1,\ldots,k-1\}$ the rightmost symbol in $c_j$ is greater (resp. greater or equal) than the rightmost symbol
in $c_{j+1}$. Note that this decomposition of $w$ exists since $w$ is not an \lps\ (resp. \rps) canonical word.

Denote by $y$ the rightmost symbol in $c_j$, by $x$ the leftmost symbol in $c_{j+1}$ that is smaller (resp. less or
equal) than $y$, and let $u_m,\ldots,u_1\in\aA_n$ and $z\in \aA_n^*$ be such that $c_{j+1}=u_m\cdots u_1xz$. (Note that
$u_m\cdots u_1$ is non-empty since $c_jc_{j+1}$ is not an \lps\ (resp. \rps) column word.) Then $yu_m\cdots u_1x$, with
$x<y\leq u_1< \cdots < u_m$ (resp. $x<y\leq u_1< \cdots < u_m$), is a factor of $w$, and hence $w$ is reducible with
respect to the relation $\rR_{\bell_n}$ (resp. $\rR_{\belr_n}$).
\end{proof}

\begin{lem}\label{lem:uniqueness_on_irreducible}
Let $n\in \mathbb{N}$.  Each $\bell_n$ class and each $\belr_n$ class has a
unique irreducible element.
\end{lem}
\begin{proof}
As the arguments follow in the same way, we just provide the proof
for the $\bell_n$ case. If $w,w'\in \aA^*_n$ are irreducible words that are in
the same $\bell_n$ class, then by the previous lemma $w$ and $w'$ are \lps\
canonical words. Since they are in the same $\bell_n$ class, that is,
$w\bellcongn w'$ (equivalently, $\bB_\ell(w)=\bB_\ell(w')$), we get
$w=\cC(\bB_\ell(w))=\cC(\bB_\ell(w'))=w'$, by Proposition~\ref{prop2.3}.
\end{proof}

Combining the previous results we
extract the following:

\begin{lem}
For any $n\in \mathbb{N}$, the \lps\ and the \rps\  canonical words are the
smallest words, with respect to the length-plus-lexicographic order, of its
$\bell_n$ and $\belr_n$ classes, respectively.
\end{lem}

\begin{thm}
  \label{thm:complete}
  For every $n\in \mathbb{N}$, the pair $(\aA_n,\rR_{\bell_n})$ is a finite and complete presentation for the \lps\
  monoid of rank $n$, $\bell_n$, whereas $(\aA_n,\rR_{\belr_n})$ is a complete presentation for the \rps\ monoid of rank
  $n$, $\belr_n$.
\end{thm}

\begin{proof}
  By Lemma~\ref{lemma1.5} the presentations $(\aA_n,\rR_{\bell_n})$ and $(\aA_n,\rR_{\belr_n})$ are noetherian. By
  Lemma~\ref{lem:uniqueness_on_irreducible}, the presentations are confluent. Thus both presentations are complete.
\end{proof}

Note that, despite not being finite, we can still algorithmically decide whether a word is a left-hand side of a rule in
$\rR_{\belr_n}$ and compute the corresponding right-hand side. Thus, we are able to compute normal forms and therefore
solve the word problem for $\belr_n$ (for more information regarding these constructions, see, for example,
\cite{MR1343236}).

\begin{cor}
	For any $n\in \mathbb{N}$, the monoids $\bell_n$ and $\belr_n$ have solvable word problem.
\end{cor}


Combining Theorem~\ref{thm:complete} with \cite[Theorem~5.3]{MR1288942}, we obtain the following corollary:

\begin{cor}
	For any $n\in\mathbb{N}$, the $\bell_n$ monoid has finite derivation type.
\end{cor}

By \cite{MR1452735} also conclude the following:

\begin{cor}
  For any $n\in\mathbb{N}$, the $\bell_n$ monoid satisfies the ho\-mo\-lo\-gi\-cal conditions left- and right-$FP_{\infty}$.
\end{cor}
%
%

\section{Automaticity and Biautomaticity}
\label{sec4}

\subsection{Preliminaries}

This subsection recall some basic definitions and necessary results from the theory of automatic and biautomatic
monoids. For further information on automatic semigroups, see \cite{MR1795250} and \cite{MR3283708}. We assume
familiarity with the theory of finite automata and regular languages, and also with the theory of transducers and
rational relations (see, for example, \cite{MR549481}).

An automatic structure for a monoid uses automata in order to compare pairs of words over an alphabet $\Sigma$. As these
pairs of words might not have the same length we have to use a padding symbol $\$$ to lengthen the shorter word to be
the same length as the longer. More precisely, let $\delta_R : \Sigma^* \times \Sigma^* \to ((\Sigma \cup \{\$\}) \times
(\Sigma \cup \{\$\}))^*$ be defined by:
	\begin{multline*}
	(u_1 \cdots u_m, v_1 \cdots v_n) \mapsto \\ \begin{cases} (u_1, v_1) \cdots
(u_m, v_n) &\mbox{if } m=n, \\
	(u_1, v_1)\cdots (u_n, v_n)(u_{n+1}, \$) \cdots (u_m, \$) &\mbox{if } m >
n,\\
	(u_1, v_1) \cdots (u_m, v_m)(\$, v_{m+1}) \cdots (\$, v_n) &\mbox{if } m <
n. \end{cases}
	\end{multline*}
	and let $\delta_L : \Sigma^* \times \Sigma^* \to ((\{\$\}\cup \Sigma) \times  (\{\$\} \cup \Sigma))^*$ be defined by
	\begin{multline*}
	(u_1 \cdots u_m, v_1 \cdots v_n) \mapsto\\ \begin{cases} (u_1, v_1) \cdots
(u_m, v_n) &\mbox{if } m=n, \\
	(u_1, \$) \cdots (u_{m-n}, \$)(u_{m-n+1}, v_1) \cdots (u_m, v_n) &\mbox{if
} m > n,\\
	(\$, v_1) \cdots (\$, v_{n-m})(u_1, v_{n-m+1}) \cdots (u_m, v_n) &\mbox{if
} m < n. \end{cases}
	\end{multline*}
where $u_i,v_i\in \Sigma$.

Let $M$ be a monoid, $\Sigma$ a finite generating set and let $L \subseteq
\Sigma^*$ a regular language such that every element of $M$ has at least one representative in
$L$.  Define the relations
\begin{align*}
	&L_a = \{(u,v): u,v\in L,\, ua =_M v\},\\
	&_aL = \{(u,v): u,v\in L,\, au =_M v\}.
\end{align*}
We say that $(\Sigma, L)$ is an \emph{automatic structure for} $M$ if $\left(L_a\right)\delta_R$ is a regular language
over $(\Sigma\cup \{\$\}) \times (\Sigma \cup\{\$\})$ for all $a \in \Sigma \cup \{\varepsilon\}$. A monoid is
\emph{automatic} if it admits an automatic structure with respect to some generating set.

We say that the pair $(\Sigma,L)$ is a \emph{biautomatic structure for} $M$ if $\left(L_a\right)\delta_R$,
$\left(_aL\right)\delta_R$, $\left(L_a\right)\delta_L$, and $\left(_aL\right)\delta_L$ are all regular languages over
$(\Sigma\cup\{\$\})\times(\Sigma\cup \{\$\})$ for all $a\in \Sigma\cup \{\varepsilon\}$ . A monoid $M$ is
\emph{biautomatic} if it admits a biautomatic structure with respect to some generating set. Clearly,
biautomaticity implies automaticity.

In order to prove that a relation $R\delta_R$ or $R\delta_L$ is regular, the following result, which is a
combination of \cite[Corollary 2.5]{MR1203822} and \cite[Proposition 4]{MR2194113}, will be useful:

\begin{prop}
  \label{prop1.9}
  If $R \subseteq \Sigma^* \times \Sigma^*$ is rational relation and there is a constant $k$ such that
  $\abs{ \abs{u} - \abs{v}}\leq k$ for all $(u, v) \in R$, then $\left(R\right)\delta_R$ and $\left(R\right)\delta_L$
  are regular.
\end{prop}

\subsection{Proving automaticity for the \rps\ monoids of finite rank}

Our goal in this subsection is to prove that the monoid $\belr_n$ is automatic for any $n\in \mathbb{N}$. The first step
will be to construct a regular language $L\subseteq \aA^*_n$ which maps surjectively (and in fact bijectively) onto
$\belr_n$. We then prove that $\left(L_{a}\right)\delta_R$ is a regular language for every
$a\in\aA_n\cup \{\varepsilon\}$ and thus show that $(\aA_n^*,L)$ is an automatic structure for $\belr_n$.  However, as
we will see, $(\aA_n^*, L)$ is only a biautomatic structure for this monoid when $n=1$. We finish the subsection
constructing a biautomatic structure for $\belr_2$.

Consider the following languages over $\aA_n$:
\begin{alignat*}{2}
	L^{(1)}&=\{n\}^*\{n-1\}^*\cdots\{2\}^*\{1\}^+\\
	L^{(2)}&=\{n\}^*\{n-1\}^*\cdots\{3\}^*\{2\}^+\\
	&\qquad\vdots\\
	L^{(n-1)}&=\{n\}^*\{n-1\}^+\\
	L^{(n)}&=\{n\}^+;
\end{alignat*}
it is immediate that these are regular languages.

For any $B\subseteq \aA_n$, with $B=\{a_1<\ldots<a_k\} \neq \emptyset$, define
\begin{equation*}
	L^B=\prod_{a_i\in B}^{\rightarrow}L^{(a_i)} = L^{(a_1)}\cdots L^{(a_k)};
\end{equation*}
this is a finite product of regular languages and thus is itself regular. Define $L^\emptyset = \{\varepsilon\}$.

Furthermore, since finite unions of regular languages are also regular,
the language
\begin{equation*}
L=\bigcup_{B\subseteq \aA_n}L^{B}
\end{equation*}
is also regular.

\begin{prop}
  \label{bijectivemap}
  For any $n\in \mathbb{N}$, the language $L\subseteq \aA^*_n$ maps bijectively onto $\belr_n$.
\end{prop}

\begin{proof}
  First note that $L^{(a_i)}$ is the set of (non-empty) weakly decreasing words ending in $a_i$, or equivalently the set
  of readings of \rps\ columns whose bottom symbol is $a_i$. Thus $L^B$ (for $B \subseteq \aA_n$) is the set of readings
  of \rps\ tableaux whose bottom row consists of the symbols in $B$ (in increasing order). Since the sets $L^B$ are
  pairwise disjoint, it follows that $L$ is in one-to-one correspondence with the set of \rps\ tableaux. Hence $L$ maps
  bijectively onto $\belr_n$.
\end{proof}

It remains to prove that for any $i\in \aA_n\cup \{\varepsilon\}$, the language $\left(L_i\right)\delta_R$ is regular.
We will first prove that the relation $L_i$ is rational. The first step is to construct some auxiliary relations.  For any
$j\in \aA_n$, let
\[
	L^{(j,\varepsilon)} = \{(u,u):u\in L^{(j)}\},
\]
and for any $i,j\in \aA_n$ such that $i\leq j$, let
\begin{equation*}
L^{(j,i)}=L^{(j,\varepsilon)} \{(\varepsilon,i)\}=\{(u,u\cdot i):u\in L^{(j)}\}.
\end{equation*}
It is immediate from the definitions that for any $j\in \aA_n$, the relation $L^{(j,\varepsilon)}$ is rational, and that
for any $i,j\in \aA_n$ with $i\leq j$, the relation $L^{(j,i)}$ is rational. Clearly, $(u,v) \in L^{(j,i)}$ if and only if
$u$ is the reading of an \rps\ column whose bottom symbol is $j$, and $v$ is the reading of the column resulting from
right-inserting a symbol $i \leq j$ (which will be inserted at the bottom of this column).

Let $j\in \aA_n$ and let $B \subseteq \aA_n$. If $j$ is greater than every symbol in $B$ (which, in particular, holds if
$B = \emptyset$), then define
\begin{align*}
	L^{(B,j)}=\left(\prod_{a_i\in B}^{\rightarrow}L^{(a_i,\varepsilon)}\right) \{(\varepsilon,j)\}.
\end{align*}
In this case, $(u,v) \in L^{(B,j)}_j$ if and only if $u$ is the column reading of an \rps\ tableau whose bottommost row
contains the symbols in $B$ (in increasing order) and $v$ is the column reading of the tableau that arises from
right-inserting the symbol $j$ (which is greater than every symbol in $B$ and is thus appended to the right side of the
bottommost row).

Otherwise $j$ is less than or equal to some symbol in $B$. Suppose $B=\{a_1<\ldots<a_k\}$ and let
$a_m=\min\{a_i\in B:a_i\geq j\}$. Partition $B$ as $B=C\cup \{a_m\}\cup E$, where $j>a$ for all $a\in C$ and $j<a$ for all $a\in
E$. Now define
\begin{align*}
L^{(B,j)}=\left(\prod_{a_i\in C}^{\rightarrow}L^{(a_i,\varepsilon)}\right) L^{(a_m,j)}\left(\prod_{a_i\in E}^{\rightarrow}L^{(a_i,\varepsilon)}\right).
\end{align*}
In this case, $(u,v) \in L^{(B,j)}_j$ if and only if $u$ is the column reading of an \rps\ tableau whose bottommost row
contains the symbols in $B$ (in increasing order) and $v$ is the column reading of the tableau that arises from
right-inserting the symbol $j$ (which is inserted at the bottom of the column ending in $a_m$).

Combining the last two paragraphs, and noting that the sets $L^{(B,j)}$ are disjoint, we see that
\[
  L_j=\bigcup_{B\subseteq\aA_n} L^{(B,j)}.
\]
Since each $L^{(B,j)}$ is a concatenation of rational relations, it too is a rational relation. Hence $L_j$ is a
rational relation.

Finally, note that
\[
  L_\varepsilon = \{(u,u):u \in L\},
\]
by Proposition~\ref{bijectivemap} and is thus a rational relation.

For any $j \in \aA_n \cup \{\varepsilon\}$, if $(u,v) \in L_j$ then $\bigl||u|-|v|\bigr| \leq 1$, since insertion into a
tableau increases the length of the column reading by $1$. Hence $L_j\delta_R$ is a regular language by
Proposition~\ref{prop1.9}. Thus,

\begin{thm}
	For any $n\in \mathbb{N}$, $(\aA_n,L)$ is an automatic structure for $\belr_n$.
\end{thm}

For $n=1$, the languages
\begin{align*}
  (_1L)\delta_R &=\{(1,1)\}^*\{(\$,1)\}, \\
  (L_1)\delta_L &=\{(\$,1)\}\{(1,1)\}^*, \\
  (_1L)\delta_L &=\{(\$,1)\}\{(1,1)\}^*,
\end{align*}
are all regular, and so $(\aA_1,L)$ is a biautomatic structure for $\belr_1$ and thus $\belr_1$ is biautomatic. However,
$(\aA_n,L)$ is not in general a biautomatic structure for $\belr_n$:

\begin{prop}
	For any $n\in \mathbb{N}$ with $n\geq 2$, the language $\left(_1L
	\right)\delta_R$ is not regular.
\end{prop}

\begin{proof}
  Let $n\in \mathbb{N}$ with $n\geq 2$. Suppose, with the aim of obtaining a contradiction, that
  \[
    \left(_1L\right)\delta_R = \{(u,v)\delta_R:u,w\in L\wedge 1u\belrcongn v\}
  \]
  is a regular language over $\left(\aA_n\cup \{\$\}\right)
  \times \left(\aA_n\cup \{\$\}\right)$. Let
  $N$ be the number of states in a finite automaton $\mathfrak{A}$ recognizing $(_1L)\delta_R$ and fix $\alpha>N$.

  For each $\beta\in \mathbb{N}$, let $u_\beta$ and $v_\beta$ be representatives of the elements of $\belr_n$ whose
  column readings are $2^\beta1^\beta$ and $1^{\beta+1}2^\beta$, respectively. Since these words are both in $L$, we
  conclude by Proposition~\ref{bijectivemap} that $u_\beta=2^\beta1^\beta$ and $v_\beta=1^{\beta+1}2^\beta$. It is
  straightforward to see that
  \[
    (u_\beta,v_\beta)\delta_R=(2^\beta1^\beta,1^{\beta+1}2^\beta)\delta_R\in
    (_1L)\delta_R.
  \]

  Since $\alpha>N$, when $\mathfrak{A}$ reads $(u_\alpha,v_\alpha)\delta_R$, it enters the same
  state after reading two different symbols $2$ of $u_\alpha$. Let $u'$ and $u'u''$ be the prefixes of $u_\alpha$
  up to and including these symbols $2$; note that $|u''| > 0$. Then $u'=2^\eta$ and $u'u''=2^\eta 2^\gamma$ for
  some $\eta, \gamma\in\mathbb{N}$ with $\gamma > 0$. Let $v'$ and $v'v''$ be the prefixes of $v_\alpha$ of
  lengths $\eta$ and $\gamma$, respectively; thus $v'=1^\eta$, $v'v''=1^\eta 1^\gamma$. Pumping
  $(u'',v'')\delta_R$ shows that
  \[
    (2^\eta 2^{2\gamma}2^{\alpha-\eta-\gamma}1^\alpha,1^\eta 1^{2\gamma}1^{\alpha+1-\eta-\gamma}2^\alpha)\delta_R \in (_1L)\delta_R.
  \]
  Hence $1\cdot 2^{\alpha+\gamma}1^\alpha \belrcongn 1^{\alpha+\gamma+1}2^\alpha$.  As $\gamma>0$, the number of
  symbols $1$ on the left side of the equality is smaller than the number of symbols $1$ on the right side of the
  equality. This is a contradiction since $\belr_n$ is multihomogeneous.
\end{proof}

We can also show that $\belr_2$ is biautomatic as follows: let
\begin{align*}
	J&=\{2^i12^k1^{j-1}: i,k\in\mathbb{N}_0, j\in \mathbb{N}\}\cup \{2^i:i\in \mathbb{N}_0\}\\
	 &=\{2\}^*\{1\}\{2\}^*\{1\}^*\cup\{2\}^*;
\end{align*}
note that $J$ is regular.

The map $\varphi:J\to\belr_2$, defined by
\[\varphi:x\mapsto
\begin{cases}
2^i1^j2^k  &\text{if $x=2^i12^k1^{j-1}$}\\	2^i  &\text{if $x=2^i$}
\end{cases}\]
is a bijection between $J$ and $\belr_2$ (viewed as column readings of tableaux). Thus, $J$ is a regular language
that maps bijectively onto $\belr_2$. Since
\begin{align*}
	J_1&=\{(2,2)\}^* \{(1,1)\}\{(2,2)\}^*\{(1,1)\}^*\{(\varepsilon,1)\} \cup
	\{(2,2)\}^*\{(\varepsilon,1)\}\\
	J_2&=\{(2,2)\}^* \{(1,1)\}\{(2,2)\}^*\{(\varepsilon,2)\}\{(1,1)\}^* \cup\{(2,2)\}^*\{(\varepsilon,2)\}\\
	_1J&=\{(\varepsilon,1)\}\{(2,2)\}^*\{(1,\varepsilon)\} \{(2,2)\}^*\{(1,1)\}^*\{(\varepsilon,1)\}\cup\{(\varepsilon,2)\}\{(2,2)\}^*\\
	_2J&=\{(\varepsilon,2)\}\{(2,2)\}^*\{(1,1)\}\{(2,2)\}^*\{(1,1)\}^*
	\{(\varepsilon,2)\}\{(2,2)\}^*
\end{align*}
we deduce that $(_iJ)\delta_R$, $(_iJ)\delta_L$, $(J_i)\delta_R$ and $(J_i)\delta_L$,
for $i\in\{1,2\}$ are regular languages. Therefore we have the following result:

\begin{prop}
	The pair $(\aA_2,J)$ is a biautomatic structure for $\belr_2$.
\end{prop}

The idea of the biautomatic structure for $\belr_2$ does not seem to generalize to higher ranks, so the following question
remains open:

\begin{qst}
	For each $n\geq 3$, is $\belr_n$ biautomatic?
\end{qst}

\subsection{Proving biautomaticity for the \lps\ monoids of finite rank}

\subsubsection{The language of representatives}

In order to prove that finite-rank \lps\ monoids are biautomatic, we will first work with a language $K$ defined over a set of generators different from $\aA_n$, and then later switch back to $\aA_n$.

Define the alphabet
\begin{align*}
	\eE_n=\{e_{\alpha}:\ \alpha=\alpha_n\cdots\alpha_1\in \aA_n^+,\
\alpha_1,\ldots,\alpha_n\in \aA_n,\ \alpha_1<\ldots<\alpha_n\}.
\end{align*}
The idea is that each \lps\ column word $\alpha$ of $\bell_n$ is represented by a single \lps\ column $e_\alpha$. Note
that since $\aA_n$ is finite and the subscript of each letter $e_\alpha$ is a strictly decreasing word, $\eE_n$ is finite.

For any $e_\alpha \in \eE_n$, let $\alpha_1$ denote the rightmost (that is, smallest) symbol in $\alpha$.  Define the
following language of representatives over $\eE_n$,
\begin{align*}
	&K=\{e_{\alpha^1}e_{\alpha^2}\cdots e_{\alpha^k}:k\in \mathbb{N}_0,
	e_{\alpha^i}\in \eE_n, \alpha_1^i\leq\alpha_1^{i+1}\text{ for all $i$}\}
\end{align*}
Notice that $e_{\alpha^1}e_{\alpha^2}\cdots e_{\alpha^k}\in K$ if and only if ${\alpha^1}{\alpha^2}\cdots{\alpha^k}$ is
the \lps\ column decomposition of the corresponding \lps\ tableau, or in other words,
${\alpha^1}{\alpha^2}\cdots {\alpha^k}=\cC(\bB_\ell({\alpha^1}{\alpha^2}\cdots {\alpha^k}))$. Thus $K$ maps bijectively to
$\bell_n$.

Observe also that in order to check that $\alpha_1^i\leq \alpha_1^{i+1}$, the automaton needs only to store the
previously read symbol. So $K$ is a regular language.

\subsubsection{Right multiplication by transducer}

We will now prove that, for any $\gamma \in \aA_n$, the relation
\begin{align*}
	&K_{e_\gamma}=\{(e_{\alpha^1}\cdots e_{\alpha^k},e_{\beta^1}\cdots
	e_{\beta^l})\in K\times K: e_{\alpha^1}\cdots e_{\alpha^k}e_{\gamma} \bellcongn e_{\beta^1}\cdots e_{\beta^l}\},
\end{align*}
is a rational relation. In order to do so, we will prove that $K_{e_\gamma}$ can be recognized by a transducer that
reads right-to-left; this will be enough to prove that $K_{e_\gamma}$ is rational since the class of rational relations
is closed under reversal \cite[pp.~65--66]{MR549481}.

We will also make use of non-determinism, in that our transducer will be able to guess the next symbol to be read. This
symbol is stored in the state of the transducer so that later it can be compared with next symbol when it is read. Then,
if the guess was correct, the transducer continues, otherwise it enters a failure state.

Thus the idea is that the transducer will read a pair of words
\begin{align*}
	&(e_{\alpha^1}e_{\alpha^2}\cdots e_{\alpha^k},e_{\beta^1}e_{\beta^2}\cdots
e_{\beta^l})\in K\times K
\end{align*}
from right to left, with the aim of checking whether the pair is in $K_{e_\gamma}$ or not.  We imagine the
transducer as reading symbols from the top tape and outputting symbols on the bottom tape. Essentially, the transducer
will perform the insertion algorithm using the alphabet $\eE_n$ as a column representation of the tableau.

The transducer stores either the symbol $e_\gamma$ or the symbol $\infty$. It non-deterministically guesses the next
symbol to be read (that is, one symbol to the left of the current symbol). Initially it stores the symbol $e_\gamma$;
the idea is that when the stored symbol is $e_\gamma$, the transducer is searching for the correct column in which to
insert $e_\gamma$. Following Algorithm~\ref{alg:Bell}, the transducer knows that it has found the correct column,
$\alpha^i$ (represented by $e_{\alpha^i}$) in which to insert $e_\gamma$ if the last symbol of $\alpha^i$ is strictly
greater than $\gamma$ and the last symbol of $\alpha^{i-1}$ (which it knows since it has non-deterministically guessed
$e_{\alpha^{i-1}}$) is less or equal than $\gamma$.  When the stored symbol is $\infty$, the transducer has completed
the insertion and simply reads the remainder symbols (which were not read yet) from the input tape and writes them on
the output tape.

Initially the transducer stores $e_\gamma$ and non-deterministically knows $e_{\alpha^k}$. If the last symbol of $\alpha^k$ is
less than or equal to $\gamma$ the transducer outputs $e_\gamma$ before reading any symbol and stores $\infty$.

When reading a symbol $e_{\alpha^i}$, the transducer non-deterministically knows $e_{\alpha^{i-1}}$, or
non-deterministically guesses that it has reached $e_{\alpha^1}$. In the later case, the transducer reads
$e_{\alpha^1}$, outputs $e_{\beta}$, with $\beta=\alpha^1\gamma$, and stores $\infty$. Otherwise, there are two
possibilities. If, when reading $e_{\alpha^i}$, the last symbol of $\alpha^{i-1}$ is less or equal than $\gamma$, then
the transducer outputs $e_{\beta}$, with $\beta=\alpha^i\gamma$, and stores $\infty$. Otherwise, it keeps the stored
symbol $e_\gamma$ and proceeds to read the symbol $e_{\alpha^{i-1}}$.

This heuristic description of the transition function corresponds to a more formal description using a finite lookup
table. The transducer performs right-insertion of a symbol $e_\gamma$ in terms of representatives in the language
$K$. Thus $K_{e_\gamma}$ is a rational relation.

\subsubsection{Left multiplication by transducer}

We want to prove that for any $\gamma \in \aA_n$, the relation
\begin{equation*}
	{}_{e_\gamma}K=\{(e_{\alpha^1}\cdots e_{\alpha^k},e_{\beta^1}\cdots
	e_{\beta^l})\in K\times K: e_{\gamma}e_{\alpha^1}\cdots e_{\alpha^k} \bellcongn e_{\beta^1}\cdots e_{\beta^l}\}
\end{equation*}
is a rational relation.

The strategy will be analogous to the one used in the right multiplication by transducer, but this time the transducer
will read pairs of words from left to right. Similarly, we will make use of the left insertion algorithm (see subsection
\ref{subsection1.3}) but using the symbols $\eE_n$. The transducer will read a pair of words
\begin{equation*}
(e_{\alpha^1}e_{\alpha^2}\cdots e_{\alpha^k},e_{\beta^1}e_{\beta^2}\cdots
e_{\beta^l})\in K\times K
\end{equation*}
from left to right with the purpose of checking whether the pair is in $_{e_\gamma}K$.

The transducer stores in its state a symbol $e_\eta$ from $\eE_n$ or a symbol $\infty$. Initially, the transducer stores
$e_\gamma$. The idea is that when the stored symbol is some $e_\eta$, the transducer will insert $\eta$ into the column
represented by the symbol it is reading from the input tape. When the stored symbol is $\infty$, the transducer has
completed the algorithm and simply reads the remaining symbols from the input tape and writes them in the output tape.

If the transducer is storing $e_\eta$ and reads $e_{\alpha^i}$, there are two possibilities:
\begin{itemize}
\item If the last symbol of $\eta$ is greater than the first symbol of $\alpha^i$, it outputs $e_{\eta\alpha^i}$ and
  stores $\infty$.
\item Otherwise, it factors $\alpha^i$ as $\alpha^i = \mu\nu$, where the length of $\mu$ is minimal such that last
  symbol of $\mu$ is greater than or equal to the last symbol of $\eta$. It outputs $e_{\eta\nu}$ and stores $e_\mu$.
\end{itemize}
If the transducer reaches the end of its input (which it can know by non-deterministically looking ahead) while storing
$e_\eta$, it outputs $e_\eta$.

Again, this heuristic description of the transition function corresponds to a more formal description using a finite
lookup table. The transducer performs left-insertion of a symbol $e_\gamma$ in terms of representatives in the language
$K$. Thus ${}_{e_\gamma}K$ is a rational relation.

\subsubsection{Deducing biautomaticity}

Therefore, for any $\gamma\in \aA_n$, both $_{e_\gamma}K$ and $K_{e_\gamma}$ are rational
relations. Define $Q\subseteq \eE_n^*\times \aA_n^*$ by
\begin{align*}
	Q=\{(e_{\alpha^1}e_{\alpha^2}\cdots
	e_{\alpha^k},\alpha^1\alpha^2\cdots
	\alpha^k):  e_{\alpha^i} \in \eE_n\}.
\end{align*}
It is easy to see that $Q$ is a rational relation. (In fact, $Q$ is the homomorphism extending $e_{\alpha} \mapsto \alpha$.)
Let
\begin{equation*}
  L=(K)Q=K\circ Q=\{v\in \aA_n^*: \exists u\in K\ (u,v)\in Q\}.
\end{equation*}
Since $K$ maps bijectively onto $\bell_n$, so does $L$. Moreover, $L$ is a regular language, because $K$ is a regular
language and a $Q$ is a rational relation and the class of regular languages is closed under applying rational
relations. For any $\gamma\in \aA_n$,
\begin{align*}
  &(u,v)\in L_\gamma \\
\iff{}& u\in L\, \wedge\, v\in L\,
\wedge\, u\gamma \bellcongn v\\
	\iff{}&\exists u',v'\in K\ (u',u)\in
Q\,\wedge\,(v',v)\in Q\,\wedge\, u'e_\gamma=_{\bell_n} v'\\
	\iff{}&\exists u',v'\in K\ (u,u')\in Q^{-1}\,\wedge\,
(u',v')\in K_{e_\gamma}\, \wedge\, (v',v)\in Q\\
	\iff{}& (u,v)\in Q^{-1}\circ K_{e_\gamma}\circ Q
\end{align*}
Therefore, as $Q^{-1}, K_{e_\gamma}$ and $Q$ are rational relations, $L_\gamma$
is a rational relation.

Similarly, from the fact that $_{e_\gamma}K$ is a rational relation, we deduce
that $_\gamma L = Q^{-1} \circ\ {_{e_\gamma}K} \circ\ Q$ is a rational relation.

Finally, since $L$ maps bijectively onto $L$,
\[
L_\varepsilon = {}_\varepsilon L = \{(u,u):u \in L\},
\]
which is a rational relation.

For any $\gamma \in \aA_n$, if $(u, v) \in L_\gamma$ or $(u,v) \in {}_\gamma L$, then $\abs{v} \leq \abs{u} + 1$ since
$u\gamma \bellcongn v$ and $\bell_n$ is multihomogeneous. By Proposition \ref{prop1.9},
$(L_\gamma) \delta_R$, $(L_\gamma) \delta_L$, $({}_\gamma L) \delta_R$ and
$({}_\gamma L) \delta_L$ are all regular. This proves the following result:

\begin{thm}
  For any $n\in \mathbb{N}$, $(\aA_n,L)$ is a biautomatic structure for $\bell_n$.
\end{thm}

\subsection{Consequences of automaticity}

The automaticity of the \rps\ and \lps\ monoids together with \cite[Corollary~3.7]{MR1795250}
imply the following result:

\begin{prop}
  For any $n\in \mathbb{N}$, the monoids $\belr_n$ and $\bell_n$ have word problem solvable in quadratic time.
\end{prop}

\bibliographystyle{alphaabbrv}
\bibliography{\jobname}

\end{document}